\documentclass{article}

\usepackage[latin1]{inputenc}
\usepackage{subfigure}
\usepackage[T1]{fontenc}
\usepackage{amsthm}
\usepackage{graphicx}
\usepackage{amsfonts}
\usepackage{amssymb} 
\usepackage{amsmath, xypic}
\usepackage[normalem]{ulem}

\newtheorem{thm}{Theorem}

\newtheorem{lem}{Lemma}
\newtheorem{prop}{Proposition}

\newtheorem*{exm}{Example}
\newtheorem*{exs}{Examples}

\newtheorem{rem}{Remark}

\newtheorem{conj}{Conjecture}

\newcommand{\al}{\alpha}
\newcommand{\be}{\beta}
\newcommand{\de}{\delta}

\newcommand{\la}{\lambda}

\renewcommand{\leq}{\leqslant}
\renewcommand{\geq}{\geqslant}
\renewcommand{\phi}{\varphi}

\DeclareMathOperator{\wt}{wt}

\newlength{\cellsz}
\newcounter{cellsize}
\newcommand{\setcellsize}[1]{%
  \setcounter{cellsize}{#1}%
  \setlength{\cellsz}{\value{cellsize}\unitlength}}%
\setcellsize{12}%
\newcommand\cellify[1]{\def\thearg{#1}\def\nothing{}%
\hbox to 0pt{{\begin{picture}(\value{cellsize},\value{cellsize})
  \put(0,0){\line(1,0){\value{cellsize}}}
  \put(0,0){\line(0,1){\value{cellsize}}}
  \put(\value{cellsize},0){\line(0,1){\value{cellsize}}}
  \put(0,\value{cellsize}){\line(1,0){\value{cellsize}}} \end{picture}
\hss}}%\fi%
\vbox to \cellsz{ \vss \hbox to \cellsz{\hss$#1$\hss} \vss}}
\newcommand\tableau[1]{\vcenter{\vbox{\let\\\cr
\baselineskip -16000pt \lineskiplimit 16000pt \lineskip 0pt
\ialign{&\cellify{##}\cr#1\crcr}}}}
\newcommand\tabl[1]{\vtop{\let\\\cr
\baselineskip -16000pt \lineskiplimit 16000pt \lineskip 0pt
\ialign{&\cellify{##}\cr#1\crcr}}}
\author{Andrei L. Kanunnikov \and Ekaterina A. Vassilieva}
\title{A recurrence formula for Jack connection coefficients}

\begin{document}

\maketitle
\begin{abstract}
This article is devoted to the study of Jack connection coefficients, a generalization of the connection coefficients of the classical commutative subalgebras of the group algebra of the symmetric group closely related to the theory of Jack symmetric functions. First introduced by Goulden and Jackson (1996) these numbers indexed by three partitions of a given integer $n$ and the Jack parameter $\alpha$ are defined as the coefficients in the power sum expansion of the Cauchy sum for Jack symmetric functions. While very little is known about them, examples of computations for small values of $n$ tend to show that the nice properties of the special cases $\alpha =1$ (connection coefficients of the class algebra) and $\alpha = 2$ (connection coefficients of the double coset algebra) extend to general $\alpha$. Goulden and Jackson conjectured that Jack connection coefficients are polynomials in $\beta = \alpha-1$ with non negative integer coefficients given by some statistics on matchings on a set of $2n$ 
elements, the so called {\em Matchings-Jack conjecture}.\\
In this paper we look at the case when two of the integer partitions are equal to the single part $(n)$ and use a framework by Lasalle (2008) for Jack symmetric functions to show that the coefficients satisfy a simple recurrence formula that makes their computation very effective and allow a better understanding of their properties. In particular we prove the Matchings-Jack conjecture in this case. Furthermore, we provide a bijective proof of the recurrence formula for $\alpha \in \{1,2\}$ using the combinatorial interpretation of the coefficients for these specific values of the Jack parameter. Finally we exhibit the polynomial properties of more general coefficients where the two single part partitions are replaced by an arbitrary number of integer partitions either equal to $(n)$ or $[1^{n-2}2]$. 
\end{abstract}

%%%%%%%%%%%%%%%%%%%%%%%%%%%%%%%%%%%%%%%%%%%%%%%%%%%%%%%%%%%%%%%%%%%%%%

\section{Introduction}
\label{sec:in}
\subsection{Integer partitions} 
For any integer $n$ we denote $[n]=\{1,\ldots,n\}$, $S_n$ the symmetric group on $n$ elements and $\la=(\la_1,\la_2,\ldots,\la_p) \vdash n$ an integer partition of $|\la| = n $ with $\ell(\la)=p$ parts sorted in decreasing order. If $m_i(\la)$ is the number of parts of $\la$ that are equal to $i$, then we may write $\la$ as $[1^{m_1(\la)}\,2^{m_2(\la)}\ldots]$ and define $Aut_{\la}=\prod_i m_i(\la)!$ and $z_\lambda =\prod_i i^{m_i(\lambda)}m_i(\lambda)!$. A partition $\lambda$ is usually represented as a Young diagram of $|\la|$ boxes arranged in $\ell(\la)$ lines so that the $i$-th line contains $\la_i$ boxes.  Given a box $s$ in the diagram of $\lambda$, let $l'(s),l(s),a(s),a'(s)$ be the number of boxes to the north, south, east, west of $s$ respectively. These statistics are called {\bf co-leglength, leglength, armlength, co-armlength} respectively. We define for some parameter $\alpha$:
\begin{align}
h_{\lambda}(\alpha)= \prod_{s\in \lambda} (\alpha a(s) + l(s) + 1), \;\;\;\;\;\;\;  h'_{\lambda}(\alpha)= \prod_{s\in \lambda} (\alpha(1+a(s)) + l(s)),
\end{align}
and denote $j_\la(\al)$ the product $j_\la(\al) = h_{\lambda}(\alpha)h'_{\lambda}(\alpha)$.\\
Finally we adopt the following notations consistent with \cite{ML08} for operations on integer partitions. We denote for a partition $\la$ containing at least one part $k$ $\la_{\downarrow(k)}$ the partition obtained from $\la$ by removing one of the parts $k$ and adding a part $k-1$ and $\la^{\uparrow(k)}$ the partition obtained from $\la$ by removing one of the parts $k$ and adding a part $k+1$. If $\la$ contains a part $k$ and a part $l$ we denote $\la_{\downarrow(k,l)}$ the partition obtained from $\la$ by removing a part $k$ and a part $l$ and adding a part $k+l-1$. Finally if $\la$ contains a part $k+l+1$ we denote $\la^{\uparrow(k,l)}$ the partition obtained from $\la$ by adding a part $k$ and a part $l$ and removing a part $k+l+1$.
\begin{equation} \label{part}
\begin{aligned} & \la_{\downarrow(k)}= \la\setminus k \cup (k-1), \quad &\la_{\downarrow(k,l)}= \la\setminus (k,l)\cup(k+l-1), \\
& \la^{\uparrow(k)}= \la\setminus k \cup (k+1), \quad & \la^{\uparrow(k,l)}= \la \setminus (k+l+1) \cup (k,l). \end{aligned}
\end{equation}

%%%%%%%%%%%%%%%%%%%%%%%%%%%%%%%%%%%%%%%%%%%%%%%%%%%%%%%%%%%%%%%%%%%%%%%%%%%

\subsection{Classes of permutations indexed by partitions}
The {\bf conjugacy classes} of the symmetric group $S_n$ are indexed by partitions of $n$ according to the cycle type of the permutations they contain. For $\la \vdash n$ we denote $C_\lambda$ the class of permutations of cycle type $\lambda$. The cardinality of the conjugacy classes is given by $|C_\la| = n!/z_\la$. We look at {\bf matchings} of the set $[n]\cup [\widehat{n}]=\{1,\ldots n, \widehat{1},\ldots, \widehat{n}\}$ which we view as fixed point free involutions in $S_{2n}$. Note that for $f,g$ fixed point free involutions of $S_{2n}$, the disjoint cycles of the product $f\circ g$ have repeated lengths i.e. $f\circ g \in C_{\la\la}$ for some $\la \vdash n$. We consider the {\bf hyperoctahedral group} $B_n$ as the centralizer of $f_\star = (1\widehat{1})(2\widehat{2})\cdots(n\widehat{n})$ in $S_{2n}$. As shown in e.g. \cite[VII.2]{IGM} the {\bf double cosets} of $B_n$ in $S_{2n}$ are also indexed by integer partitions of $n$. We denote by $K_\lambda$ the double coset indexed by $\la \vdash n$ 
consisting of all the 
permutations $\omega$ of $S_{2n}$ such that $f_\star \circ\omega\circ f_\star\circ\omega^{-1}$ belongs to $C_{\la\la}$. According to this definition $K_{\la}=B_n\omega B_n$ for any $\omega$ in $K_{\la}$ and, in particular, $B_n=K_{[1^n]}$. We have \cite[VII.2, (2.3)]{IGM} $|B_n| = 2^nn!$ and $|K_\la| = |B_n|^2/(2^{\ell(\la)}z_\la)$.

%%%%%%%%%%%%%%%%%%%%%%%%%%%%%%%%%%%%%%%%%%%%%%%%%%%%%%%%%%%%%%%%%%%%%%

\subsection{Symmetric functions}
Let $\Lambda$ be the ring of symmetric functions. Denote $m_\lambda(x)$ the monomial symmetric function indexed by $\lambda$ on indeterminate $x$, $p_\lambda(x)$ and $s_\lambda(x)$  the power sum and Schur symmetric functions respectively. Whenever the indeterminate is not relevant we shall simply write $m_\la$, $p_\la$ and $s_\la$. Let $\langle \cdot\, ,\cdot \rangle$ be the scalar product on $\Lambda$ such that the power sum symmetric functions verify $\left <p_\la,p_\mu \right> = z_\la\delta_{\la\mu}$ where $\delta_{\la\mu}$ is the Kronecker delta. The Schur symmetric functions $s_\la$ are characterized by the fact that they form an orthogonal basis of $\Lambda$ for $\langle \cdot\, ,\cdot \rangle$  and the transition matrix between Schur and monomial symmetric functions is upper triangular. 

The {\bf zonal polynomials} $Z_\la$ constitute another important basis of $\Lambda$ directly linked with the theory of the {\bf zonal spherical functions}. Zonal polynomials verify the same properties as the $s_\la$ if the scalar product is replaced by $\langle \cdot,\cdot \rangle_2$ with $\langle p_\la,p_\mu \rangle_2 = 2^{\ell(\la)}z_\la\delta_{\la\mu}$. In the general case, using an additional parameter $\alpha$, Henry Jack \cite{HJ} introduced the bases of {\bf Jack symmetric functions} $J^\alpha_\la$ orthogonal for the alternative scalar product $\langle \cdot\, ,\cdot \rangle_\alpha$ defined by $\langle p_\la,p_\mu \rangle_\alpha = \alpha^{\ell(\la)}z_\la\delta_{\la\mu}$. We use the normalization of Jack symmetric functions such that $[m_\la]J^\alpha_\la=h_\la(\alpha)$. As a result we have 
\begin{equation} \label{norm12} J^1_\la = h_\la(1)s_\la \quad \text{ and } \quad J^2_\la =Z_\la.\end{equation} Moreover as shown by  Stanley in \cite{S89}, 
\begin{equation}
\label{eq : jla}\langle J^\al_\la,J^\al_\mu \rangle_\alpha = j_\la(\al)\delta_{\la\mu}.
\end{equation}

%%%%%%%%%%%%%%%%%%%%%%%%%%%%%%%%%%%%%%%%%%%%%%%%%%%%%%%%%%%%%%%%%%%%%%

\subsection{Classical connection coefficients and their combinatorial interpretation} \label{ccc}
By abuse of notation let $C_\lambda$ (resp. $K_\lambda$) also represent the formal sum of its elements in the group algebra $\mathbb{C} S_{n}$ (resp. $\mathbb{C} S_{2n}$). So $\{C_\la\mid \la \vdash n\}$ forms a basis of the {\bf class algebra} i.e. the center of $\mathbb{C} S_n$ and $\{K_\la\mid \la \vdash n\}$ forms a basis of the {\bf double coset algebra} i.e. the commutative subalgebra of $\mathbb{C}S_{2n}$ identified as the {\bf Hecke algebra of the Gelfand pair} $(S_{2n},B_n)$.
For integer $s\geq 2$ and partitions $\la^1,\ldots,\la^s \vdash n$, we define the {\bf connection coefficients of the class algebra} $c^{\la^1}_{\lambda^2,\ldots,\la^s}$ and the {\bf connection coefficients of the double coset algebra} $b^{\la^1}_{\lambda^2,\ldots,\la^s}$ by
\begin{equation}
 c^{\la^1}_{\lambda^2,\ldots,\la^s} = [C_{\la^1}]\prod_{i \geq 2}C_{\lambda^i}, \;\;\;\;\; b^{\la^1}_{\lambda^2,\ldots,\la^s} = [K_{\la^1}]\prod_{i \geq 2}K_{\lambda^i}.
\end{equation}
From a combinatorial point of view $c^{\la^1}_{\lambda^2,\ldots,\la^s}$ (resp. $b^{\la^1}_{\lambda^2,\ldots,\la^s}$) is the number of ways to write a given permutation $\sigma_1$ of $C_{\la^1}$ (resp. $K_{\la^1}$) as the ordered product of $s-1$ permutations $\sigma_2\circ\ldots\circ\sigma_s$ where $\sigma_i$ is in $C_{\lambda^i}$ (resp. $K_{\la^i}$).\\ 
The coefficients $c^\la_{\mu\nu}$ and $b^{\la}_{\mu\nu}$ also admits a nice interpretation in terms of graphs on $2n$ vertices (see e.g. \cite{GJ96, S92}). For a given partition $\la=(\la_1,\ldots,\la_p)$ of integer $n$, consider the graph $G$ on $2n$ vertices consisting of $p$ cycles of lengths $2\la_1, \ldots, 2\la_p$. A matching in $G$ is a set of edges without common vertices that contains all the vertices of $G$. Coloring successively the edges of the cycles of $G$ in gray and black colors, we get two matchings: {\bf g} (gray edges) and {\bf b} (black edges).  We call such a~two-colour graph induced by $\la$ {\bf a~$\la$-graph}. Label the vertices of $G$ by $[n]\cup[\widehat{n}]$ such that edges $\{i,\widehat{i}\}$ are gray and the vertices of the $i$-th cycle are successively labelled $$\sum_{j=1}^{i-1}\la_j+1,\widehat{\sum_{j=1}^{i-1}\la_j+1},\sum_{j=1}^{i-1}\la_j+2,\widehat{\sum_{j=1}^{i-1}\la_j+2},\ldots,\sum_{j=1}^{i}\la_j,\widehat{\sum_{j=1}^{i}\la_j}.$$ We call such labeling {\bf 
canonical} and a matching in which all edges are of kind $\{i,\widehat{j}\}$, $1\leq i,j\leq n$, {\bf bipartite}. 

\begin{exm} Figure \ref{la_g} depicts a canonically labeled $\la$-graph for $\la=(3,2,2,1)$. 
\end{exm}
\begin{figure}[htbp]
\begin{center}
 \includegraphics[scale=0.7]{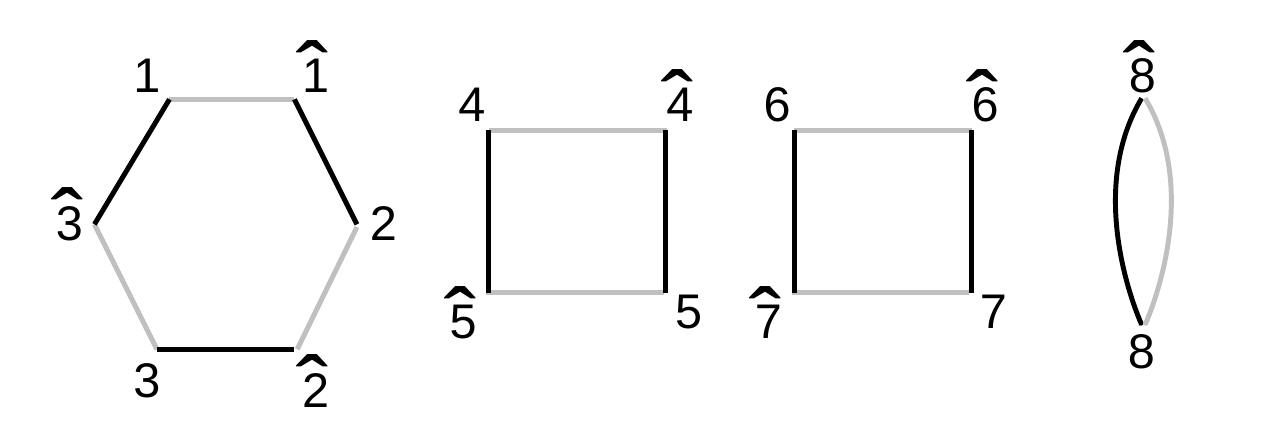}\caption{the union of the black and gray matchings forms the $(3,2,2,1)$-graph.}
 \label{la_g}
 \end{center}
 \end{figure}
Denote $\tilde{b}^{\la}_{\mu\nu}$ the quantity $\tilde{b}^{\la}_{\mu\nu}:=b^{\la}_{\mu\nu}/|B_n|.$ We have the following proposition.
\begin{prop}[\cite{GJ96}, proposition 4.1; \,\cite{S92}, Lemma 3.2] \label{combint}  Let $\la \vdash n$ and $G$ the induced $\la$-graph. Defining the matchings {\bf b} and  {\bf g} in $G$ as above, one has
\begin{itemize}
\item[(1)] $\tilde{b}^\la_{\mu \nu}$ is the number of matchings $\de$ such that the graphs {\bf b} $\cup$ $\de$ and {\bf g} $\cup$ $\de$ are respectively a $\mu$-graph and a $\nu$-graph;
\item[(2)] $c^\la_{\mu \nu}$ is the number of {\bf bipartite} matchings $\de$ such that the graphs {\bf b} $\cup$ $\de$ and {\bf g} $\cup$ $\de$ are respectively a $\mu$-graph and a $\nu$-graph.
\end{itemize}
\end{prop}
In what follows we call a matching $\de$ such that both {\bf b} $\cup$ $\de$ and {\bf g} $\cup$ $\de$ are $2n$-cycles ($(n)$-graphs) a {\bf good matching}. Denote by $\mathcal{G}(\la)$ the set of all good matchings of the canonically labeled $\la$-graph $G$. Due to Proposition \ref{combint} 
\begin{equation} \label{combintn} \tilde{b}^\la_{nn}=|\mathcal{G}(\la)| \quad \text{ and } \quad c^\la_{nn}=|\{\de\in \mathcal{G}(\la)\mid \de \text { is bipartite}\}. \end{equation}
\begin{exs}
1. One can see $\la$-graphs for $\la$ with $|\la|\leq 2$ on Figure \ref{la12}. So we have
\centerline{$\begin{aligned} &\tilde{b}^1_{11}=1,\\ &c^1_{11}=1, \end{aligned}$ \quad
$\begin{aligned} &\tilde{b}^2_{22}=1,\\ &c^2_{22}=0, \end{aligned}$ \quad
$\begin{aligned} &\tilde{b}^{(1,1)}_{22}=2,\\ &c^{(1,1)}_{22}=1.\end{aligned}$}

 \begin{figure}[htbp] \begin{center} \includegraphics[scale=0.6]{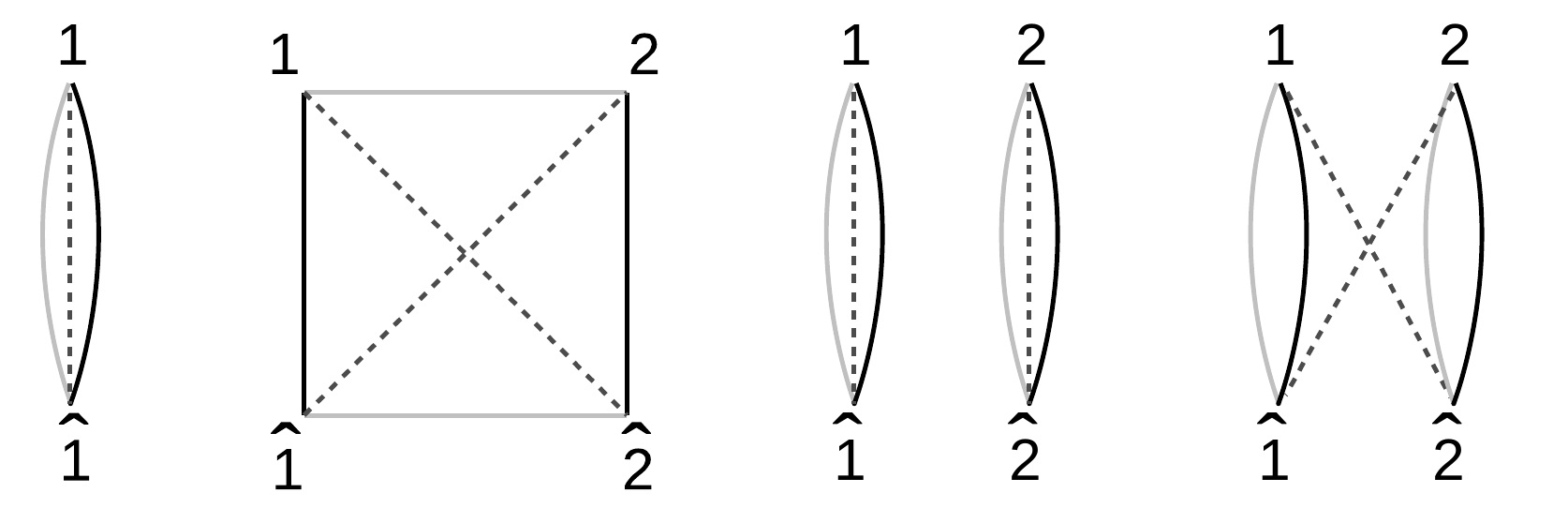} \caption{Good matchings for $|\la|\leq 2$.} \label{la12}  \end{center}\end{figure}

2. Figure \ref{good3} depicts the $4$ good matchings in the case $\la=(3)$. Only the leftmost one is bipartite. As a result
$\tilde{b}^3_{33}=4$ and $c^3_{33}=1$.
\begin{figure}[htbp]\begin{center} \includegraphics[scale=0.6]{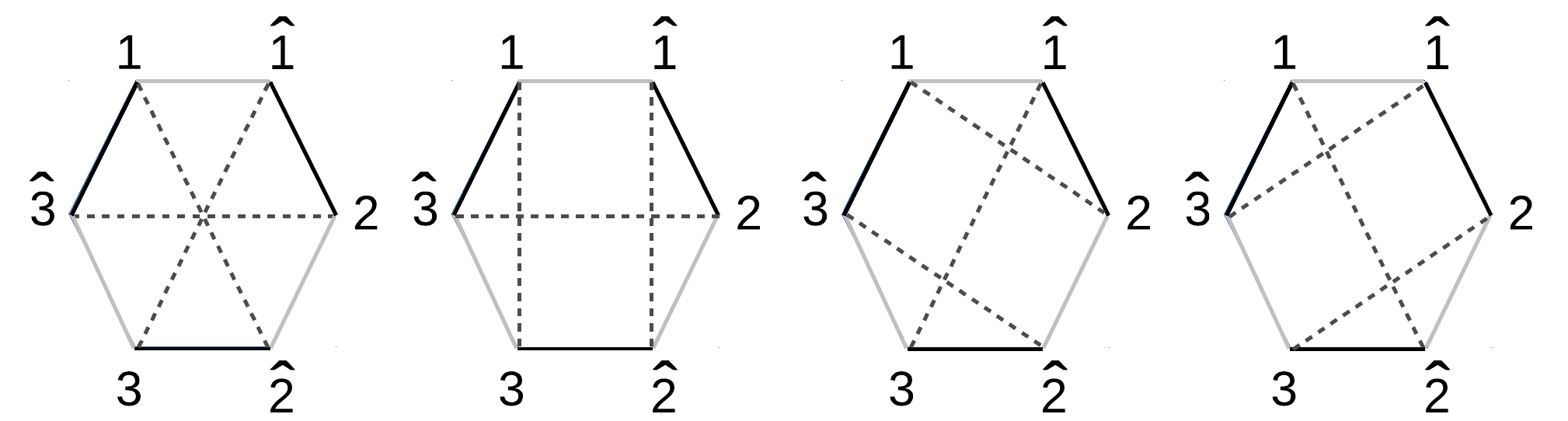} \caption{Good matchings for $\la=(3)$.} \label{good3}\end{center} \end{figure}
\end{exs}

\subsection{Jack connection coefficients and the Matchings-Jack conjecture}
It is easy to show (see e.g. \cite{GJ96}, \cite{HSS}, \cite{V2014}) that connection coefficients are linked to some extended Cauchy sums for Schur and zonal symmetric functions. In particular one has:
\begin{align}
\label{eq : sumc}&\sum_{\la,\mu,\nu \vdash n}z_\la^{-1}c^{\la}_{\mu\nu}p_\la(x)p_\la(y)p_\la(z) = \sum_{\gamma \vdash n}h_\gamma(1)s_\gamma(x)s_\gamma(y)s_\gamma(z)\\
\label{eq : sumb}&\sum_{\la,\mu,\nu \vdash n}2^{-\ell(\la)}z_\la^{-1}\frac{b^{\la}_{\mu\nu}}{|B_n|}p_\la(x)p_\la(y)p_\la(z) = \sum_{\gamma \vdash n}\frac{Z_\gamma(x)Z_\gamma(y)Z_\gamma(z)}{\langle Z_\gamma,Z_\gamma\rangle_{2}}
\end{align}
In the general case, Goulden and Jackson \cite{GJ96} considered the coefficients $a_{\mu\nu}^\la(\al)$ in the power sum expansions of similar sums for Jack symmetric functions
\begin{equation}
\label{eq : def}
\sum_{\la, \mu, \nu \vdash n}\al^{-\ell(\la)}z_\la^{-1}a_{\mu\nu}^\la(\al)p_\la(x)p_\mu(y)p_\nu(z) = \sum_{\gamma \vdash n}\frac{J^{\al}_\gamma(x)J^{\al}_\gamma(y)J^{\al}_\gamma(z)}{\langle J^{\al}_\gamma,J^{\al}_\gamma\rangle_{\al}}
\end{equation} 
In view of Equations (\ref{norm12}), (\ref{eq : sumc}), (\ref{eq : sumb}) and (\ref{eq : def})
\begin{equation} \label{cb} a_{\mu\nu}^\la(1) = c_{\mu\nu}^\la \quad \text{ and } \quad a_{\mu\nu}^\la(2) = \frac{1}{|B_n|}b^{\la}_{\mu,\nu}=\tilde{b}^\la_{\mu \nu}.\end{equation}
Computations of $a^{\la}_{\mu\nu}(\alpha)$ for all $\la,\mu,\nu\vdash n \leq 8$ (see \cite{GJ96}) showed that the $a^{\la}_{\mu\nu}(\alpha)$ are polynomials in $\beta = \al -1$ with non negative integer coefficients and of degree at most $n-\min\{\ell(\mu),\ell(\nu)\}$. Goulden and Jackson conjectured this property for arbitrary $\la,\mu,\nu$. 
Moreover, following the combinatorial interpretation in Proposition \ref{combint}, they also suggest the stronger {\bf Matchings-Jack conjecture.}

\begin{conj}[\cite{GJ96}, conjecture 4.2] For $\la,\mu,\nu\vdash n$
$$a_{\mu\nu}^\la(\be+1) = \sum_{\de}\be^{\wt_{\la}(\de)}$$
where the summation is over all matchings of item (1) in Proposition \ref{combint}, and 
$\wt_\la(\de)\in \{0,1,\ldots, n-\min\{\ell(\mu),\ell(\nu)\}\}$, $\wt_\la(\de)=0 \iff \de$ is bipartite.
\end{conj} 
They proved this conjecture in the cases $\la =[1^n]$ and $\la = [1^{n-2}2^1]$. Following \cite{V2014}, we use the terminology of {\bf Jack connection coefficients} for $a_{\mu\nu}^\la(\al)$ (although the normalization in \cite{V2014} differs from the one of this paper). Our work is mainly devoted to the case $\mu=\nu=(n)$. We show various results for Jack connection coefficients in this case, including the proof of the Matchings Jack conjecture. 

\subsection{Main results}

Our proof of Goulden and Jackson's conjecture is based on the recurrence formula (\ref{eq : rec}) which derives from the
following theorem proved in Section \ref{sec : alg}.\\
For simplicity, we denote $n$ instead of $(n)$ in the indices of the coefficients. 

%Suppose $n =pq$ for some integers $p$ and $q$. We look at the case when $\mu = (n)$, $\nu = [q^{p}]$ and denote $a_{n,p\times q}^\la(\alpha)$ instead of $a_{(n)[q^{p}]}^\la(\alpha)$ for simplicity.  We show the following theorem.
\begin{thm}
\label{thm : rec}
For positive integer $n$ and integer partitions $\nu \vdash n$ and $\la \vdash n+1$, the Jack connection coefficients defined in Equation (\ref{eq : def}) satisfy the following formula:
\begin{align} 
\nonumber \sum_{i:\,m_{i-1}(\nu) \neq 0}i(m_i(\nu)+1)&a_{n+1,\nu^{\uparrow(i-1)}}^\la(\al) = \sum_{i =1}^{\ell(\la)}\la_i \left[ (\al-1)(\la_i-1)a_{n,\nu}^{\la_{\downarrow(\la_i)}}(\al)\right.\\
\label{eq : recnu}&+\sum_{d=1}^{\la_i-2}a_{n,\nu}^{\la^{\uparrow(\la_i-1-d,d)}}(\al)+\al\sum_{j\neq i}\left.\la_ja_{n,\nu}^{\la_{\downarrow(\la_i,\la_j)}}(\al)\right].
\end{align}
\end{thm}

%\begin{rem} Of course, the sum $\sum_{d=1}^{\la_i-2}$ is equal to $0$ if $\la_i=1$ or $2$.
%\end{rem}

In the special case $\nu = (n)$ the formula (\ref{eq : recnu}) becomes recursive. Besides, as shown in 
Section \ref{sec : add}, the expression in brackets does not depend on $i$. This remarkable property allows us to derive the following theorem. 
\begin{thm}
\label{thm : reca}
For integer $n$ and partition $\la \vdash n+1$, the~Jack connection coefficients verify the following recurrence formula for any $i\in \{1,\ldots, \ell(\la)\}$:
\begin{align} 
\nonumber a_{n+1,n+1}^\la(\al) = (\al-1)(&\la_i-1)a_{nn}^{\la_{\downarrow(\la_i)}}(\al)\\
\label{eq : rec}&+\sum_{d=1}^{\la_i-2}a_{nn}^{\la^{\uparrow(\la_i-1-d,d)}}(\al)+\al\sum_{j\neq i}\la_ja_{nn}^{\la_{\downarrow(\la_i,\la_j)}}(\al).
\end{align} 
\end{thm}
\begin{rem}
One can use the fact that Equation (\ref{eq : rec}) is true for any $i$ to derive additional formulas. For instance, let $\mu$ be a partition of integer $n-1$, the~following expression holds
\begin{equation*}
a_{nn}^{\mu\cup(1)}(\al) = \al (n-1) a_{n-1,n-1}^{\mu}(\al).
\end{equation*}
One can iterate this relation for any partition $\la$ of $n$ that we write $\la = \mu\cup (1^{m_1(\la)})$:
\begin{equation*}
a_{nn}^{\la}(\al) = \al^{m_1(\la)} \frac{(n-1)!}{(n-m_1(\la)-1)!} a_{n-m_1(\la),n-m_1(\la)}^{\mu}(\al).
\end{equation*}
Equation (\ref{eq : rec}) also proves for $\mu \vdash n-2$
\begin{equation*}
a_{nn}^{\mu\cup (2)}(\al) = \al(\al-1) (n-2) a_{n-2,n-2}^{\mu}(\al)+\al\sum_{j}\mu_ja_{n-1,n-1}^{\mu^{\uparrow(\mu_j)}}.
\end{equation*}
\end{rem}
Theorem \ref{thm : reca} allows us to prove the Matchings-Jack conjecture in the case $\mu=\nu=(n)$.

\begin{thm} \label{thm : MJC}
Let $\la$ be a partition of $n$ and $G$ a canonically labelled $\la$-graph. Then there exists a~function $\wt_{\la}\colon \mathcal{G}(\la)\to \{0,1,\ldots,n-1\}$ such that 
$$ a^{\la}_{nn}(\be+1)=\sum_{\de\in\mathcal{G}(\la)} \be^{\wt_{\la}(\de)}$$ and $\wt_{\la}(\de)=0 \iff \de$ is bipartite.

As a consequence, the quantity $a^\la_{nn}(\be+1)$ is a nonnegative integer polynomial in $\be$ with
constant term $c^\la_{nn}=a^\la_{nn}(1)$ and sum of coefficients $\tilde{b}^\la_{nn}=a^\la_{nn}(2)$. Besides
$$[\be^{n-1}]a^\la_{nn}(\be+1) = (n-1)!.$$
\end{thm}

%\subsection{General Jack connection coefficients}

For indeterminate $x^1, x^2,\ldots,x^s$ define the more general Jack connection coefficients
\begin{equation}
\label{eq : defGen}
\sum_{\la^i \vdash n}\al^{-\ell(\la^1)}z_{\la^1}^{-1}a_{\la_2,\ldots,\la_s}^{\la^1}(\al)\prod_ip_{\la^i}(x^i) = \sum_{\gamma \vdash n}\frac{\prod_iJ^{\al}_\gamma(x^i)}{\langle J^{\al}_\gamma,J^{\al}_\gamma\rangle_{\al}}
\end{equation} 
and denote for any integer partition $\la \vdash n$ and integer $l,r\geq 0$ $$a_\la^{l,r}(\al) = a_{\underbrace{(n),\ldots,(n)}_{l},\underbrace{[1^{n-2}2],\ldots,[1^{n-2}2]}_{r}}^{\la}(\al).$$
We show some polynomial properties for the coefficients $a_\la^{l,r}(\al)$.
\begin{thm}
\label{thm : GenCoeff}
Let $D(\al)$ denote the Laplace Beltrami operator (see Equation (\ref{eq : D})), $\Delta_0(\al) =p_1/\al$ and $\Delta_k(\al) = [D(\al),\Delta_{k-1}(\al)]$. The coefficients $a_\la^{l,r}(\al)$ defined above verify
\begin{equation}
\al^{-\ell(\la)}|C_\la|a_\la^{l,r}(\al) = [p_\la]{D(\al)}^r{\Delta_l(\al)}^{n-1}({p_1}/{\al})
\end{equation}
and for $l \geq 2$, $|C_\la|a_\la^{l,r}(\al)$ is a polynomial in $\al$ with integer coefficients of degree at most $(n-1)(l-1)+r$. Furthermore
\begin{equation*}
[\al^i]|C_\la|a_\la^{l,r} (\al) = (-1)^{(l-1)(n-1)+r+\ell(\la)-1}[\al^{(l-1)(n-1)+r+\ell(\la)-1-i}]|C_\la|a_\la^{l,r} (\al).
\end{equation*}
\end{thm}
Equation (\ref{eq : rec}) admits a nice combinatorial interpretation in terms of permutations and graphs in the special cases $\al =1,\,2$ that provided us with the intuition for the general case. These bijections are described in Section \ref{sec : bij}. Using the theory of Jack symmetric functions, we prove Theorem \ref{thm : rec} in Section \ref{sec : alg}. Section \ref{sec : add} provides additional results including the proof of Theorem \ref{thm : reca}, the Matchings-Jack conjecture of Theorem \ref{thm : MJC} and the proof of Theorem \ref{thm : GenCoeff}. 

%%%%%%%%%%%%%%%%%%%%%%%%%%%%%%%%%%%%%%%%%%%%%%%%%%%%%%%%%%%%%%%%%%%%%%

\section{Background}
Our work is mainly related to two areas of combinatorics, namely the computation of classical connection coefficients of group algebra of the symmetric group and the study of the coefficients in the power sum expansions of Jack symmetric functions. Relevant background is described in the two following sections. 
\subsection{Computation of connection coefficients}
Except for special cases no closed formulas are known for the coefficients  $c^{\la^1}_{\lambda^2,\ldots,\la^s}$ and $b^{\la^1}_{\lambda^2,\ldots,\la^s}$. Using an inductive argument B\'{e}dard and Goupil  \cite{BG} first found a formula for $c^n_{\lambda,\mu}$ in the case $\ell(\la)+\ell(\mu)=n+1$, which was later reproved by Goulden and Jackson \cite{GJ92} via a bijection with a set of ordered rooted bicolored trees. Later, using characters of the symmetric group and a combinatorial development, Goupil and Schaeffer \cite{GS} derived an expression for connection coefficients of arbitrary genus as a sum of positive terms (see Biane \cite{PB} for a succinct algebraic derivation; and Poulalhon and Schaeffer \cite{PS}, and Irving \cite{JI} for further generalizations). 
Closed form formulas of the expansion of the generating series for the $c^{(n)}_{\lambda^2,\ldots,\la^s}$ (for general $s$) and $b^{(n)}_{\lambda,\mu}$ in the monomial basis were provided by Morales and Vassilieva and Vassilieva in \cite{MV09}, \cite{MV11}, \cite{V2012} and \cite{V2013}. Papers \cite{MV11} and \cite{V2013} use the topological interpretation of $c^{(n)}_{\lambda^2,\ldots,\la^s}$ and $b^{(n)}_{\lambda,\mu}$ in terms of unicellular locally orientable hypermaps and constellations of given vertex degree distribution. Jackson (\cite{DMJ}) computed a general expression for the generating series of the $\sum_{\ell(\la_i) = p_i}c^{\la^1}_{\lambda^2,\ldots,\la^s}$ in terms of some explicit polynomials.

\subsection{Jack characters}
The link between Schur (resp. zonal) polynomials and irreducible characters of the symmetric group (resp. zonal spherical functions) is given by the decomposition of the $s_\la$ (resp. $Z_\la$) in the power sum basis
\begin{align}
\label{eq : characters}s_\la &= \sum_{\mu \vdash n}z_\mu^{-1}\chi^\la_\mu p_\mu,\\
\label{eq : spherical}Z_\la &= \frac{1}{| B_n |}\sum_{\mu \vdash n}\varphi^\la_\mu p_\mu
\end{align}
where $\chi^\la_\mu$ is the value of the irreducible character of the symmetric group $\chi^\la$ indexed by integer partition $\la$ at any element of $C_\mu$ and $\varphi^\la_\mu = \sum_{\omega \in K_\mu}\chi^{2\la}(\omega)$. The value of the zonal spherical function indexed by $\lambda$ of the Gelfand pair $(S_{2n}, B_n)$ at the elements of the double coset $K_\mu$ is given by $| K_\mu |^{-1} \varphi^{\la}_\mu$.
Using the fact that both Schur and zonal symmetric functions are special cases of Jack symmetric functions, it is natural to focus on a more general form of Equations (\ref{eq : characters}) and (\ref{eq : spherical}). Formally, let $\theta_\mu^\la(\alpha)$ denote the coefficient of $p_\mu$ in the expansion of $J^\alpha_\la$ in the power sum basis:
\begin{equation}
\label{eq : jackchar}J^\alpha_\la = \sum_{\mu}\theta_\mu^\la(\alpha)p_\mu.
\end{equation}
According to Equations (\ref{eq : characters}) and (\ref{eq : spherical}), up to a normalization factor, the $\theta_\mu^\la(\alpha)$'s coincide with the irreducible characters of the symmetric group and the zonal spherical functions in the cases $\alpha =1$ and $\alpha = 2$. In the general case Do\l \k{e}ga and F\'eray \cite{DF} named them {\bf Jack characters}.\\
The coefficients in the power sum expansion of Jack symmetric functions have received significant attention over the past decades. As an example, Hanlon conjectured a first combinatorial interpretation for them in \cite{H88} in terms of digraphs. Stanley proved various results for these coefficients in \cite{S89}. For instance for simple values of $\mu$ the following formulas are fulfilled:
\begin{align}
&\theta_{[1^n]}^\la(\alpha) = 1,\\
&\label{eq : theta_21}\theta_{[1^{n-2}2^1]}^\la(\alpha) =  \sum_{s \in \la}(\al a'(s)-l'(s)),\\
&\label{eq : theta_n}\theta_{(n)}^\la(\alpha) = \prod_{s \in \la \setminus\{(1,1)\}}(\al a'(s)-l'(s)).
\end{align}

More recent works by Lassalle (see e.g. \cite{ML08} and \cite{ML09}) reconsidered these coefficients as generalizations of irreducible characters of the symmetric group and conjectured various polynomial properties in $\al$ for the $\theta_\mu^\la(\alpha)$. Furthermore Lassalle showed in \cite{ML08} some recursive formulas for $\vartheta_\mu^\la = z_\mu\theta_{\mu,1^{n-k}}^\la(\alpha)$ where $\la \vdash n$ and $\mu \vdash k$ with $m_1(\mu) = 0$. Namely,
\begin{align*}
&\sum_{r,s}rs(m_r(\mu)(m_s(\mu) -\delta_{rs})\vartheta^\la_{\mu_{\downarrow(r,s)}} +(\al-1)\sum_{r}^{\bullet}r(r-1)\vartheta^\la_{\mu_{\downarrow(r)}}+\\
&\al\sum_{r}^{\bullet}rm_r(\mu)\sum_{i=1}^{r-2}\vartheta^{\la}_{\mu^{\uparrow(i,r-i-1)}} = -(n+k)\vartheta^{\la}_{\mu}+\al\sum_{i=1}^{\ell(\la)+1}c_i(\la)(\la_i-(i-1)/\al)^2\vartheta_{\mu}^{\la^{(i)}}
\end{align*}
where the symbol $\sum^{\bullet}$ indicates some additional factors in limit conditions (see \cite{ML08}), the numbers $c_i(\la)$ are defined in Equation (\ref{eq : cdef}) and $\la^{(i)}$ is defined in Section \ref{sec : lbogbcpf}.% In the special case $\la = [q^p]$,

While the formulas in Theorems \ref{thm : rec} and \ref{thm : reca} share some similarities with the above recurrence formula, they cannot be derived from it.

Do\l \k{e}ga and F\'eray \cite{DF} proved that Jack characters are polynomials in $\al$ with rational coefficients. Together with P. \'Sniady \cite{DFS} they conjectured an expression involving a measure of "non-orientability" of  locally orientable hypermaps (an expression also introduced in \cite{GJ96}). Do\l \k{e}ga and F\'eray also proved in \cite{DF14} that the $a^\la_{\mu\nu}(\al)$ are polynomials in $\al$ with {\bf rational} coefficients.
\begin{rem}
Using Equations (\ref{eq : jackchar}) and (\ref{eq : jla}), one can reformulate Jack connection coefficients as
\begin{equation}
a^{\la^1}_{\la_2,\ldots,\la_s}(\al) = \al^{\ell(\la^1)}z_{\la^1}\sum_{\gamma \vdash n}\frac{\prod_i \theta^{\gamma}_{\la^i}(\al)}{j_\gamma(\al)}.   
\end{equation}
We use this formulation in the following sections.
\end{rem}

%%%%%%%%%%%%%%%%%%%%%%%%%%%%%%%%%%%%%%%%%%%%%%%%%%%%%%%%%%%%%%%%%%%%%%%%%%%%%%%%%%%%%%%%%%%%%%%%%%%%%%%%%%%%%%%%%%%%

\section{Graph interpretation}
\label{sec : bij}

In this section we obtain recursive formulas for classical connection coefficients $\tilde{b}^\la_{nn}$ and $c^\la_{nn}$ (see equalities (\ref{cb})) using their 
combinatorial consideration described by Equation (\ref{combintn}). The recursion is obtained by removing and replacing edges
in graphs. Note that this operation is also considered in \cite{DFS} in terms of maps (graphs embedding in an orintable or a non-orientable 
surface).

\subsection{Reduction of $\la$-graphs by replacing of edges} \label{sec : red}

Consider a $\la$-graph $G$ consisiting of cycles $C_{2\la_1}$, $C_{2\la_2}, \ldots$

Let $a$ and $b$ are non-neighbor vertices of $G$. Let $a_1, a_2$ are neighbors of $a$ and $b_1,b_2$ are neighbors of $b$ such that 
$\{a,a_1\}$, $\{b,b_1\}$ are gray edges and $\{a,a_2\}$, $\{b,b_2\}$ are black edges. Replace the edges $\{a,a_1\}$ and $\{b,b_1\}$
by the gray edge $\{a_1,b_1\}$ and, similarly, replace the edges $\{a,a_2\}$ and $\{b,b_2\}$ by the black edge $\{a_2,b_2\}$.
Then the $\la$-graph $G$ transforms into a $\la'$-graph $G'$ where
\begin{equation} \label{red} \la':=\left\{ \begin{aligned} &\la_{\downarrow{(\la_i)}} &\text{ if }  a,b \in C_{2\la_i} & 
\text{ and the number of vertices} \\ & & & \text{ between $a$ and $b$ is odd,}\\
& \lambda^{\uparrow{(\lambda_i-1-d,d)}} & \text{ if }   a,b \in C_{2\la_i} & \text { and the number of vertices} \\ & & &
\text{ between $a$ and $b$ is equal to $2d-2$,} \\  
& \la_{\downarrow{(\lambda_i, \lambda_j)}} &\text{ if } a \in C_{2\la_i},\;\, &b\in C_{2\la_j},\,j\ne i.  \end{aligned} \right. 
\end{equation}

We show all three cases on Figures \ref{FIGa}--\ref{FIGc} where we label vertices of the cycle $C_{2\la_i}$ 
by $1,\widehat{1},\ldots, \la_i, \widehat{\la_i}$ and put $a=1$.

\begin{figure}[htbp]\begin{center} \includegraphics[scale=0.6]{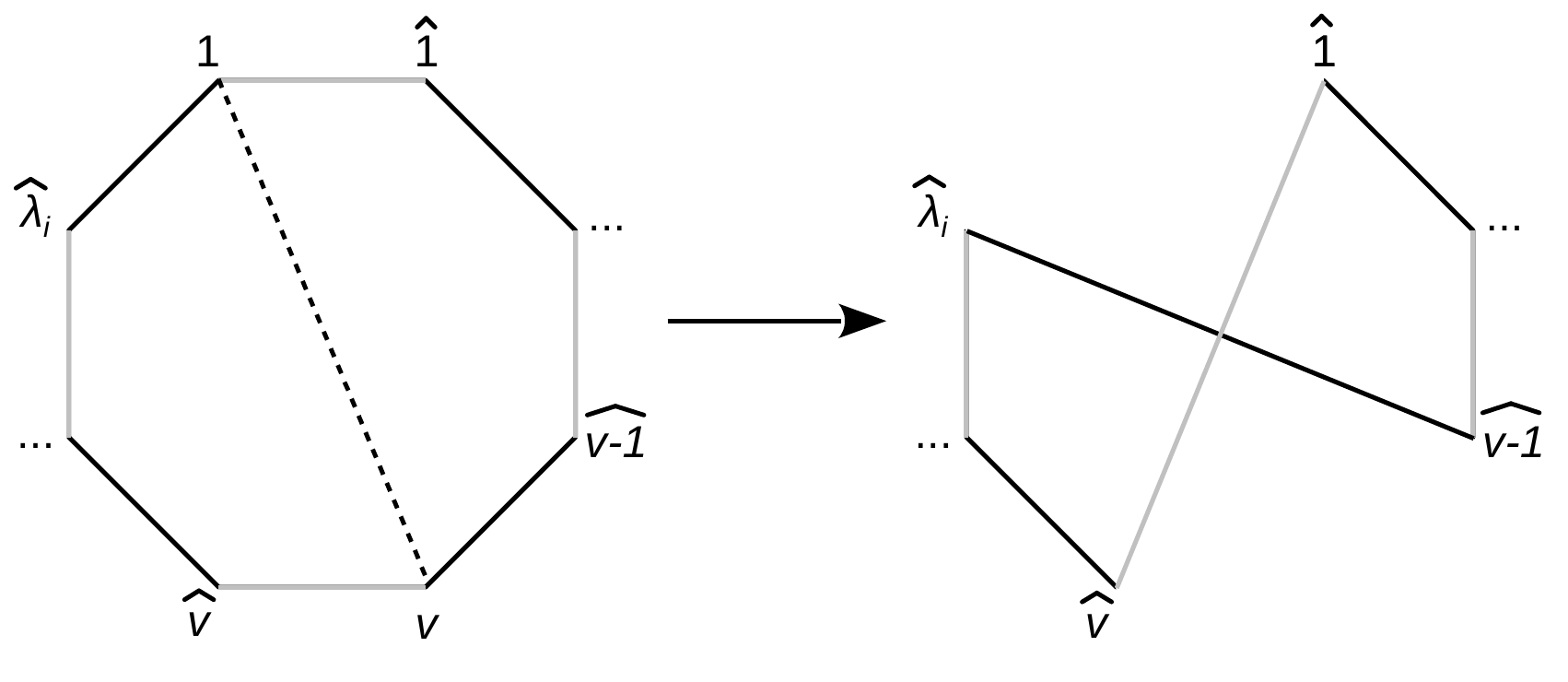} \caption{The replacing of edges transforms a $\la$-graph into a $\la_{\downarrow(\la_i)}$-graph.} \label{FIGa}\end{center} \end{figure}

\begin{figure}[htbp]\begin{center} \includegraphics[scale=0.6]{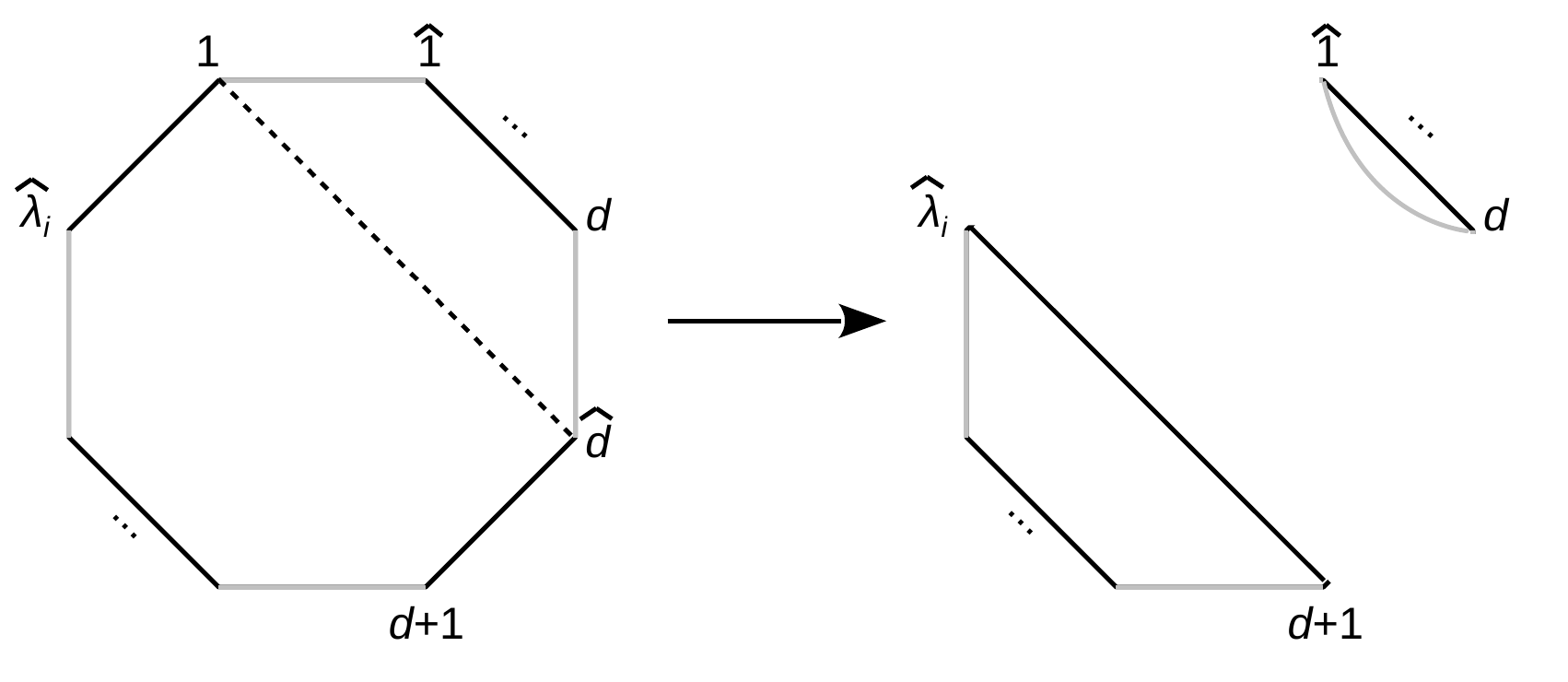} \caption{The replacing of edges transforms a $\la$-graph into a $\la^{\uparrow(\la_i-1-d,d)}$-graph.} \label{FIGb} \end{center}\end{figure}

\begin{figure}[htbp] \begin{center} \includegraphics[scale=0.6]{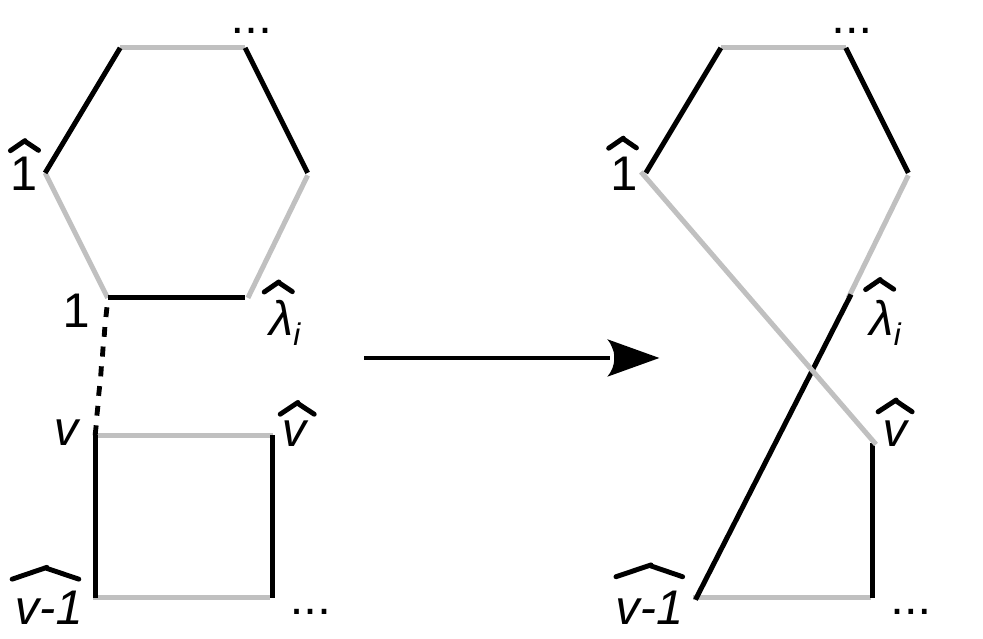}\caption{The replacing of edges transforms a $\la$-graph into a $\la_{\downarrow(\la_i,\la_j)}$-graph.}  \label{FIGc} \end{center} \end{figure}

\begin{rem} These three cases correspond to classification of edges of maps considering in \cite{DFS}. One can translate the terminology of pure graphs used by us to terms of maps. ly, the~case~1 corresponds to twisted edges from \cite{DFS}, the~case~2 corresponds to straight edges and the case~3 corresponds to interface edges. Besides bipartite matchings correspond to the case of orientable surfaces (in~particaular, twisted edges are not possible in this case).
\end{rem}

\subsection{Proof of Theorem \ref{thm : rec} in the case $\al \in \{1,2\}$}

\begin{thm} \label{thm : comb}
Let $\la$ be a partition of integer $n+1\geq 2$ and $G$ a canonically labeled $\la$-graph.
Then coefficients (\ref{cb}) satisfy the following recurrence formulas for any $i\in \{1,\ldots, \ell(\la)\}\colon$  
\begin{align} \label{comb2}  &\tilde{b}^{\lambda}_{n+1,n+1}=
(\lambda_i-1)\tilde{b}^{\lambda_{\downarrow{(\la_i)}}}_{nn}+\sum_{d=1}^{\lambda_i-2}\tilde{b}^{\lambda^{\uparrow{(\lambda_i-1-d,d)}}}_{nn}+
2\sum_{j\ne i} \lambda_j \tilde{b}^{\lambda_{\downarrow{(\lambda_i, \lambda_j)}}}_{nn}, \\
 \label{comb1}  &c^{\lambda}_{n+1,n+1}=\sum_{d=1}^{\lambda_i-2}c^{\lambda^{\uparrow{(\lambda_i-1-d,d)}}}_{nn}+\sum_{j\ne i} \lambda_j c^{\lambda_{\downarrow{(\lambda_i, \lambda_j)}}}_{nn}. \end{align} 
\end{thm}

\begin{proof} One can consider that vertices of $2\la_i$-cycle are labeled by elements 
$1,\widehat{1},\ldots,$ $\la_i,\widehat{\la_i}$. In order to prove formula (\ref{comb2}) fix any vertex of $2\la_i$-cycle, 
say $1$. All $\tilde{b}^{\lambda}_{n+1,n+1}$ good matchings $\de$ are splited into three groups depending on the position of 
the vertex $v$ which is connected with $1$ in $\de$. Write $\delta=\{\{1,v\}\}\cup \delta'$ and denote by $G'$ a $\la'$-graph
obtained from $G$ by the replacing of edges described in Section \ref{sec : red} ($\la'$ is defined as in (\ref{red}) where
$a=1$ and $b=v$). It is clear that the matching $\de$ is good for $G$ if and only if the matching $\de'$ is good for $G'$.
Indeed, $(n+1)$-cycles {\bf g}$\cup\de$ and {\bf b}$\cup\de$ transform into $n$-cycles {\bf g}$'\cup\de'$ and {\bf b}$'\cup\de'$
respectively, where {\bf g}$'$ and {\bf b}$'$ are the sets of gray and black edges of the new $\la'$-graph $G'$.
Consider three cases as above.
\vskip 3 pt

{\bf Case 1:} $v\in \{2,\ldots,\la_i\}$ (this case is possible if $\la_i\geq 2$). See Figure \ref{FIGa}. 

As $\la'=\la_{\downarrow{(\la_i)}}$ and $v$ runs the set of $\la_i-1$ vertices in this case, then the number of such matchings 
$\de$ is equal to $(\la_i-1) \tilde{b}^{\la_{\downarrow(i)}}_{nn}$.
\vskip 4 pt
{\bf Case 2:} $v=\widehat{d}$ where $d\in\{2,\ldots, \lambda_i-1\}$  (this case is possible if $\la_i\geq  3$). 
See Figure \ref{FIGb}. 
In this case, $\la'=\lambda^{\uparrow{(\lambda_i-1-d,d)}}$, so the number of such matchings $\de$ for fix $d$ is equal 
to $\tilde{b}^{\lambda^{\uparrow{(\lambda_i-1-d,d)}}}_{nn}$.
\vskip 4 pt
{\bf Case 3:}  $v$ belongs to $2\lambda_j$-cycle where $j\ne i$. See Figure \ref{FIGc}. 

Now $\la'=\la_{\downarrow{(\lambda_i, \lambda_j)}}$ and as $v$ runs the set of $2\lambda_j$ vertices then the number 
of such matchings $\de$ is equal to $2\lambda_j \tilde{b}^{\lambda_{\downarrow{(\lambda_i, \lambda_j)}}}_{nn}.$

Summing numbers of good matchings in all cases we obtain formula (\ref{comb2}). 

\vskip 5 pt
 The proof of formula (\ref{comb1}) is similar. But in view of biparticity of good matchings in this case we have two distinct:
\begin{itemize} 
\item the case 1 is not possible;

\item in the case 3 $v\in [\widehat{n}]$ so the factor before the sum $\sum_{j\ne i} c^{\la_{\downarrow{(\lambda_i, \lambda_j)}}}_{nn}$ is equal to $\la_j$ (not $2\la_j$),
\end{itemize}
  due to item (2) of Proposition \ref{combint}.
\end{proof}

\begin{rem} Of course, one can prove Theorem \ref{thm : comb} in terms of permutations without using graphs. But this way is much longer even in the simple case of coefficients $c^{\la}_{nn}$. Despite this, it is useful to clarify the reduction of $n+1$ to $n$ in terms of permutations. Exactly, let $\sigma=\tau\rho$ is a factorization of a permutation 
$\sigma\in S_{n+1}$ as the product of two $(n+1)$-cycles. Suppose that the cycle type of $\sigma$ is $\la$ and element $n+1$ belongs to the cycle of size $\la_i$. Denote by $\tau'$ and $\rho'$ the $n$-cycles obtained from $\sigma$ and $\tau$ by removing $n+1$. One can prove that the cycle type of $\sigma'\tau'\in S_n$ is $\la^{\uparrow(\la_i-1-d,d)}$ for some $d\in \{1,\ldots,\la_i-2\}$ or $\la_{\downarrow(\la_i,\la_j)}$ for some $j\ne i$.    
\end{rem}

\begin{exs} Show that $c_{55}^5= 2c^{(3,1)}_{44}+c_{44}^{(2,2)}$ in terms of permutations.

a) Show all ways to write a permutation $(123)(4) \in S_4$ as a product of two $4$-cycles and corresponding representations of a permutation $(12345)\in S_5$ as the product of two $5$-cycles: 
$$\begin{aligned} &(123)(4)=(1432)(1324) & \leftrightarrow & \quad (12345)=(15432)(13524), \\ & (123)(4)=(1324)(1342) & \leftrightarrow & \quad(12345)=(13254)(13542),
\\ & (123)(4)=(1342)(1432) & \leftrightarrow & \quad(12345)=(13542)(14352). \end{aligned}$$

b) Show all ways to write Permutation $(12)(34) \in S_4$ as the product of two $4$-cycles and the corresponding factorizations Permutation $(12345)\in S_5$ as a product of two $5$-cycles: 
$$\begin{aligned} & (12)(34)=(1324)(1324) & \leftrightarrow & \quad (12345)=(15324)(13254),\\ & (12)(34)=(1423)(1423) & \leftrightarrow & \quad(12345)=(14253)(14253).\end{aligned}$$
So $c_{55}^5=2\cdot 3+2=8$.
\end{exs}

%%%%%%%%%%%%%%%%%%%%%%%%%%%%%%%%%%%%%%%%%%%%%%%%%%%%%%%%%%%%%%%%%%%%%%%%%%%%%%%%%%%%%%%%%%%%%%%%%%%%%%%%%%%%%%%%%%%%
%%%%%%%%%%%%%%%%%%%%%%%%%%%%%%%%%%%%%%%%%%%%%%%%%%%%%%%%%%%%%%%%%%%%%%%%%%%%%%%%%%%%%%%%%%%%%%%%%%%%%%%%%%%%%%%%%%%%

\section{Proof of Theorem \ref{thm : rec}}
\label{sec : alg}
\subsection{Laplace Beltrami operator, generalized binomial coefficients and Pieri formula}
\label{sec : lbogbcpf}
In order to prove Theorem \ref{thm : rec} for general $\alpha$ we need to recall a few properties about Jack symmetric functions. \\
For indeterminate $x = (x_1,x_2,\ldots)$ define the {\bf Laplace Beltrami operator} by
\begin{equation}
\label{eq : D}
D(\al) = \frac{\al}{2}\sum_{i}x_i^2\frac{\partial^2}{\partial x_i^2}+\sum_{i}\sum_{j\neq i}\frac{x_ix_j}{x_i-x_j}\frac{\partial}{\partial x_i}
\end{equation}
We have the two following classical properties.
\begin{lem}
The Jack's symmetric functions are eigenfunctions of $D(\al)$:
\begin{equation}
\label{eq : eigen}
D(\al)J_\la^\al =\theta_{[1^{|\la|-2}2^1]}^\la(\al)J_\la^\al
\end{equation}
\end{lem}
\begin{lem}
\label{lem : dp} Operator $D(\al)$ can be expressed as:
\begin{equation*}
D(\al) = (\al-1)N + \al U + S,
\end{equation*}
where for $\la \vdash n$:
\begin{align}
\label{eq : n}N &=  \frac{1}{2}\sum_{i}i(i-1)p_i\frac{\partial}{\partial p_i},\\
\label{eq : defU} U &=\frac{1}{2}\sum_{i,j}ijp_{i+j}\frac{\partial}{\partial p_i}\frac{\partial}{\partial p_j},\\
 S & = \frac{1}{2}\sum_{i,j}(i+j)p_ip_j\frac{\partial}{\partial p_{i+j}},
\end{align}
\end{lem}
Then, for two partitions, $\mu \subseteq \la$ the generalized binomial coefficients $\binom{\la}{\mu}$ are defined through the relation
\begin{equation*}
\frac{J_\la(1+x_1, 1+x_2,\ldots)}{J_\la(1,1,\ldots)} = \sum_{\mu \subseteq \la}\binom{\la}{\mu}\frac{J_\mu(x_1,x_2,\ldots)}{J_\mu(1,1,\ldots)}.
\end{equation*}
For $\rho \vdash n+1$ and integer $1\leq i \leq \ell(\rho)$ define the partition $\rho_{(i)}$ of $n$ (if it exists) obtained by replacing $\rho_i$ in $\rho$ by $\rho_i-1$ and keeping all the other parts as in $\rho$. Similarly for $\gamma \vdash n$ and integer $1\leq i \leq \ell(\gamma)+1$ we define the partition $\gamma^{(i)}$ of $n+1$ (if it exists) obtained by replacing $\gamma_i$ in $\gamma$ by $\gamma_i+1$ and keeping all the other parts as in $\gamma$. Define also the numbers $c_i(\gamma)$ as
\begin{equation}
\label{eq : cdef} c_i(\gamma) = \alpha\binom{\gamma^{(i)}}{\gamma}\frac{j_\gamma(\al)}{j_{\gamma^{(i)}}(\al)}.
\end{equation}
As a result of this definition, one has:
\begin{equation}
\label{eq : c}
c_i(\rho_{(i)}) = \alpha\binom{\rho}{\rho_{(i)}}\frac{j_{\rho_{(i)}}(\al)}{j_\rho(\al)}
\end{equation}
In \cite{ML89} Lassalle show that the Pieri formula for Jack symmetric functions is equal to
\begin{equation*}
p_1J_\gamma^\al = \sum_{i=1}^{\ell(\gamma)+1}c_i(\gamma)J_{\gamma^{(i)}}.
\end{equation*}
Additionally Lassalle showed \cite{ML89} that
\begin{equation}
\label{eq : p1p}p_1^\perp J_\rho^\al = \al\frac{\partial}{\partial p_1}J_\rho^\al = \al \sum_{i=1}^{\ell(\rho)}\binom{\rho}{\rho_{(i)}}J_{\rho_{(i)}}.
\end{equation}
Following \cite{ML08}, we introduce the two conjugate operators $E_2$ and  $E_2^{\perp}$
\begin{align}
\label{eq : e} &E_2 = [D(\al),p_1/\al] = \sum_{k\geq 1}kp_{k+1}\frac{\partial}{\partial p_k},\\
&E_2^{\perp} = [p_1^\perp/\al,D(\al)] = \sum_{k\geq 1}(k+1)p_k\frac{\partial}{\partial p_{k+1}}.
\end{align}
Using indeterminate $x = (x_1,x_2,\ldots)$, on can express $E_2$ (see \cite{ML08}) as:
\begin{equation}
\label{eq : ealt} E_2 = \sum_{i}x_i^2\frac{\partial}{\partial x_i}.
\end{equation}

%%%%%%%%%%%%%%%%%%%%%%%%%%%%%%%%%%%%%%%%%%%%%%%%%%%%%%%%%%%%%%%%%%%%%%

\subsection{Relating two series for Jack symmetric functions}
\label{sec : rtsfjsf}
We're now ready to show our main formula. As a first step this section is devoted to the proof of the following theorem.
\begin{thm}\label{thm : JJJ} Let $x$ and $y$ be two indeterminate. The following relation holds:
\begin{equation}
\sum_{\rho \vdash n+1}\frac{\theta^\rho_{n+1}J_\rho^\al(x)E_2^\perp J_\rho^\al(y)}{j_\rho(\al)} = \sum_{\gamma \vdash n}\frac{\theta^\gamma_{n}J_\gamma^\al(y)[D(\al),E_2]J_\gamma^\al(x)}{j_\gamma(\al)}.
\end{equation}
\end{thm}
\noindent From Equation (\ref{eq : p1p}), (\ref{eq : c}) and (\ref{eq : eigen}), one gets for any integer partition $\rho \vdash n+1$
\begin{align}
\frac{E_2^\perp J_\rho^\al(y)}{j_\rho(\al)} &=\frac{1}{j_\rho(\al)} [p_1^\perp/\al,D(\al)]J_\rho^\al(y) \\
&= \frac{1}{\al j_\rho(\al)}p_1^\perp D(\al)J_\rho^\al(y) - \frac{1}{\al}D(\al)\left(\sum_{i=1}^{\ell(\rho)}c_i(\rho_{(i)})\frac{J_{\rho_{(i)}}^\al(y)}{j_{\rho_{(i)}}(\al)}\right) \\
&=\frac{1}{\al}\sum_{i=1}^{\ell(\rho)}c_i(\rho_{(i)})\left(\theta^\rho_{[1^{n-1}2]}(\al)-\theta^{\rho_{(i)}}_{[1^{n-2}2]}(\al)\right)\frac{J_{\rho_{(i)}}^\al(y)}{j_{\rho_{(i)}}(\al)}.
\end{align}
Multiplying both sides by $\theta_{n+1}^\rho(\al) J_\rho^\al(x)$ and summing over all partitions $\rho$ of $n+1$ yields
\begin{align}
\nonumber &\sum_{\rho \vdash n+1}\frac{\theta^\rho_{n+1}(\al)J_\rho^\al(x) E_2^\perp J_\rho^\al(y)}{j_\rho(\al)} =\\
\label{eq : 40}&\frac{1}{\al}\sum_{\rho \vdash n+1} \sum_{i=1}^{\ell(\rho)}c_i(\rho_{(i)})\left(\theta^\rho_{[1^{n-1}2]}(\al)-\theta^{\rho_{(i)}}_{[1^{n-2}2]}(\al)\right)\frac{\theta^\rho_{n+1}(\al)J_\rho^\al(x)J_{\rho_{(i)}}^\al(y)}{j_{\rho{(i)}}(\al)}.
\end{align}
Reorganizing the summation indices on the RHS gives
\begin{align}
\nonumber &\sum_{\rho \vdash n+1}\frac{\theta^\rho_{n+1}(\al)J_\rho^\al(x) E_2^\perp J_\rho^\al(y)}{j_\rho(\al)} =\\
&\frac{1}{\al}\sum_{\gamma \vdash n} \sum_{i=1}^{\ell(\gamma)+1}c_i(\gamma)\left(\theta^{\gamma^{(i)}}_{[1^{n-1}2]}(\al)-\theta^{\gamma}_{[1^{n-2}2]}(\al)\right)\frac{\theta^{\gamma^{(i)}}_{n+1}(\al)J_{\gamma^{(i)}}^\al(x)J_{\gamma}^\al(y)}{j_{\gamma}(\al)}.
\end{align}
We prove the following lemma
\begin{lem}
\label{lem : theta}
For $\gamma \vdash n$ and integer $i$ the following relation is fulfilled:
\begin{equation}
\theta_{n+1}^{\gamma^{(i)}}(\al) = \theta_{n}^\gamma(\al)\left(\theta^{\gamma^{(i)}}_{[1^{n-1}2]}(\al)-\theta^{\gamma}_{[1^{n-2}2]}(\al)\right).
\end{equation}
\end{lem}
\begin{proof}
According to Equations (\ref{eq : theta_n}) and (\ref{eq : theta_21}) 
\begin{align}
&\theta_{n+1}^{\gamma^{(i)}}(\al) = \prod_{s \in \gamma^{(i)}\setminus (1,1)}(\al a'(s) - l'(s)),\\
&\theta_{[1^{n-1}2]}^{\gamma^{(i)}}(\al) = \sum_{s \in \gamma^{(i)}}(\al a'(s) - l'(s)).
\end{align}
As per its definition, the Young diagram of $\gamma^{(i)}$ is obtained from the one of $\gamma$ by adding a box in the position $(i,\gamma_i+1)$. But $a'(s)$ and $l'(s)$ are only functions of the number of boxes to the west and the north of the considered square $s$. As a result the value of $(\al a'(s) - l'(s))$ for a given $s$ in $\gamma^{(i)}$ is the same as for the corresponding box in $\gamma$ except for the special case $s = (i,\gamma_i+1)$. Therefore
\begin{align*}
\nonumber \theta_{n+1}^{\gamma^{(i)}}(\al) &= \left(\prod_{s \in \gamma\setminus (1,1)}(\al a'(s) - l'(s))\right)\bigl(\al a'(i,\gamma_i+1) - l'(i,\gamma_i+1)\bigr)\\
&=\theta_{n}^\gamma(\al)\left(\theta^{\gamma^{(i)}}_{[1^{n-1}2]}(\al)-\theta^{\gamma}_{[1^{n-2}2]}(\al)\right).
\end{align*}
This is the desired formula.
\end{proof}
\noindent Thanks to Lemma \ref{lem : theta} we get
\begin{align}
\nonumber &\sum_{\rho \vdash n+1}\frac{\theta^\rho_{n+1}(\al)J_\rho^\al(x) E_2^\perp J_\rho^\al(y)}{j_\rho(\al)} =\\
\label{eq : Ti-T}&\frac{1}{\al}\sum_{\gamma \vdash n}\frac{\theta_{n}^\gamma(\al)J_{\gamma}^\al(y)}{j_{\gamma}(\al)} \sum_{i=1}^{\ell(\gamma)+1}c_i(\gamma)\left(\theta^{\gamma^{(i)}}_{[1^{n-1}2]}(\al)-\theta^{\gamma}_{[1^{n-2}2]}(\al)\right)^2J_{\gamma^{(i)}}^\al(x).
\end{align}
Finally noticing the relation for integers $a,b$ 
\begin{align}
\nonumber\sum_{i=1}^{\ell(\gamma)+1}c_i(\gamma)\left(\theta^{\gamma^{(i)}}_{[1^{n-1}2]}(\al)\right)^a&\left(\theta^{\gamma}_{[1^{n-2}2]}(\al)\right)^bJ_{\gamma^{(i)}}^\al(x)\\
\nonumber&=D(\al)^a\left(\theta^{\gamma}_{[1^{n-2}2]}(\al)\right)^b\sum_{i=1}^{\ell(\gamma)+1}c_i(\gamma)J_{\gamma^{(i)}}^\al(x)\\
\nonumber&=D(\al)^a\left(\theta^{\gamma}_{[1^{n-2}2]}(\al)\right)^bp_1J_\gamma^\al(x)\\
\label{eq : DapDb}& =D(\al)^ap_1D(\al)^bJ_\gamma^\al(x)
\end{align}
gives the desired result. Indeed
\begin{align*}
\nonumber \sum_{\rho \vdash n+1}&\frac{\theta^\rho_{n+1}(\al)J_\rho^\al(x) E_2^\perp J_\rho^\al(y)}{j_\rho(\al)} \\
\nonumber &=\frac{1}{\al}\sum_{\gamma \vdash n}\frac{\theta_{n}^\gamma(\al)J_{\gamma}^\al(y)}{j_{\gamma}(\al)}\left (D(\al)^2p_1 -2D(\al)p_1D(\al)+p_1D(\al)^2\right)J_{\gamma}^\al(x)\\
\nonumber &=\sum_{\gamma \vdash n}\frac{\theta_{n}^\gamma(\al)J_{\gamma}^\al(y)}{j_{\gamma}(\al)}[D(\al),[D(\al),p_1/\al]]J_{\gamma}^\al(x)\\
&=\sum_{\gamma \vdash n}\frac{\theta_{n}^\gamma(\al)J_{\gamma}^\al(y)}{j_{\gamma}(\al)}[D(\al),E_2]J_{\gamma}^\al(x).
\end{align*}

%%%%%%%%%%%%%%%%%%%%%%%%%%%%%%%%%%%%%%%%%%%%%%%%%%%%%%%%%%%%%%%%%%%%%%

\subsection{Recurrence relation between connection coefficients}
For $\la \vdash n+1$ and $\nu \vdash n$ extracting the coefficient in $p_\la(x)p_{\nu}(y)$ in Theorem \ref{thm : JJJ} yields
\begin{align*}
\nonumber \sum_{i:\,m_{i-1}(\nu) \neq 0}i(m_i(\nu)+1)\sum_{\rho \vdash n+1}&\frac{\theta^\rho_{n+1}(\al)\theta^\rho_{\nu^{\uparrow(i-1)}}(\al)\theta^\rho_{\la}(\al)}{j_\rho(\al)} =\\
 &\sum_{\gamma \vdash n}\frac{{\theta^\gamma_{n}(\al)}{\theta^\gamma_{\nu}(\al)}}{j_\gamma(\al)}[p_\la]\left([D(\al),E_2]J_\gamma^\al\right).
\end{align*}
We get the Jack connection coefficients $a_{n+1,\nu^{\uparrow(i-1)}}^\la(\al)$ multiplying the LHS by $\al^{\ell(\la)}z_\la$.
\begin{align}
\nonumber \sum_{i:\,m_{i-1}(\nu) \neq 0}i(m_i(\nu)+1)&a_{n+1,\nu^{\uparrow(i-1)}}^\la(\al)=\\
\label{eq : lala}& \al^{\ell(\la)}z_\la\sum_{\gamma \vdash n}\frac{{\theta^\gamma_{n}(\al)}{\theta^\gamma_{\nu}(\al)}}{j_\gamma(\al)}[p_\la]\left([D(\al),E_2]J_\gamma^\al\right).
\end{align}
The final step is to compute the remaining non explicit part $[p_\la]\left([D(\al),E_2]J_\gamma^\al\right)$. We show the following lemma:
\begin{lem}
The Lie hook of operators $D(\al)$ and $E_2$ is given by
\begin{align*}
\nonumber [D(\al),E_2] =(\al-1)\sum_{i\geq1}(i-1)^2p_i\frac{\partial}{\partial p_{i-1}}+ &\sum_{i,j\geq1}(i+j-1)p_ip_j\frac{\partial}{\partial p_{i+j-1}}\\
&+\al\sum_{i,j\geq1}ijp_{i+j+1}\frac{\partial}{\partial p_{i}}\frac{\partial}{\partial p_{j}}.
\end{align*}
\end{lem}
\begin{proof}
Recall Equations (\ref{eq : D}) and (\ref{eq : ealt}). We have
\begin{align*}
[D(\al),E_2] &= \frac{\al}{2}\sum_{i}\left[x_i^2\frac{\partial^2}{\partial x_i^2},x_i^2\frac{\partial}{\partial x_i}\right]+\sum_{i\neq j}\left[\frac{x_ix_j}{x_i-x_j}\frac{\partial}{\partial x_i},x_i^2\frac{\partial}{\partial x_i}+x_j^2\frac{\partial}{\partial x_j}\right]\\
&= \frac{\al}{2}\sum_{i}\left(2x_i^3\frac{\partial^2}{\partial x_i^2}+2x_i^2\frac{\partial}{\partial x_i}\right) +\sum_{i\neq j}\frac{2x_i^2x_j}{x_i-x_j}\frac{\partial}{\partial x_i}.
\end{align*}
Then, for any integer partition $\la$, one can easily check:
\begin{align}
\label{eq : lolo}&\sum_i x_i^3\frac{\partial^2}{\partial x_i^2}p_\la = \sum_{i=1}^{\ell(\la)}\la_i(\la_i-1)\frac{p_{\la_i+1}}{p_{\la_i}}p_\la+\sum_{i\neq j}\la_i\la_j\frac{p_{\la_{i}+\la_{j}+1}}{p_{\la_i}p_{\la_j}}p_\la,\\
\label{eq : lili}&\sum_i x_i^2\frac{\partial}{\partial x_i}p_\la = \sum_{i=1}^{\ell(\la)}\la_i\frac{p_{\la_i+1}}{p_{\la_i}}p_\la.
\end{align}
Assume $\la$ is the single part partition $(a)$ for some integer $a$. One has
\begin{align*}
\sum_{i\neq j}\frac{2x_i^2x_j}{x_i-x_j}\frac{\partial}{\partial x_i}p_a &= 2a\sum_{i\neq j}\frac{x_i^{a+1}x_j}{x_i-x_j}\\
&=a\sum_{i\neq j}\frac{x_i^{a+1}x_j}{x_i-x_j}+\frac{x_j^{a+1}x_i}{x_j-x_i}\\
&=a\sum_{i\neq j}x_ix_j\frac{x_i^{a}-x_j^a}{x_i-x_j}\\
&=a\sum_{i\neq j}\sum_{k=1}^{a}x_i^{a+1-k}x_j^{k}\\
&=a\left(\sum_{k=1}^{a}\left(\sum_{i}x_i^{a+1-k}\right)\left(\sum_{j}x_j^{k}\right)-\sum_{k=1}^{a}x_i^{a+1-k}x_i^k\right)\\
&=a\sum_{k=1}^{a}p_{a+1-k}p_k-a^2p_{a+1}.
\end{align*}
As a result, for any partition $\la$
\begin{align}
\label{eq : lele}\sum_{i\neq j}\frac{2x_i^2x_j}{x_i-x_j}\frac{\partial}{\partial x_i}p_\la = \sum_{i=1}^{\ell(\la)}\la_i\sum_{k=1}^{\la_i}\frac{p_{\la_i+1-k}p_k}{p_{\la_i}}p_\la-\sum_{i=1}^{\ell(\la)}\la_i^2\frac{p_{\la_i+1}}{p_{\la_i}}p_\la.
\end{align}
Combining Equations (\ref{eq : lolo}), (\ref{eq : lili}) and (\ref{eq : lele}) gives the desired result.
\end{proof}

\noindent Using this lemma we can derive
\begin{align*}
\nonumber [p_\la]&\left([D(\al),E_2]J_\gamma^\al\right) = (\al-1)\sum_{i}(i-1)^2(m_{i-1}(\la)+1)\theta^{\gamma}_{\la_{\downarrow(i)}},\\
\nonumber &+\al\sum_{i,d}(i-1-d)j(m_{i-1-d}(\la)+1)(m_{d}(\la)+1-\delta_{i-d-1,j})\theta^{\gamma}_{\la^{\uparrow(i-1-d,d)}},\\
&+\sum_{i,j}(i+j-1)(m_{i-1+j}(\la)+1)\theta^{\gamma}_{\la_{\downarrow(i,j)}}.
\end{align*}
But
\begin{align}
\nonumber &\al^{\ell(\la)}z_\la(i-1)^2(m_{i-1}(\la)+1) = (i-1)im_i(\la)z_{\la_{\downarrow(i)}}\al^{\ell(\la_{\downarrow(i)})},\\
\nonumber &\al^{\ell(\la)}z_\la\al(i-1-d)j(m_{i-1-d}(\la)+1)(m_{d}(\la)+1-\delta_{i-d-1,d}),\\
\nonumber &\phantom{\al^{\ell(\la)}z_\la\al(i-1-d)jllllllllllll}=im_i(\la)z_{\la^{\uparrow(i-1-d,d)}}\al^{\ell(\la^{\uparrow(i-1-d,d)})},\\
\nonumber  &\al^{\ell(\la)}z_\la(i+j-1)(m_{i+j-1}(\la)+1) = \al ijm_i(\la)(m_j(\la)-\delta_{ij})z_{\la_{\downarrow(i,j)}}\al^{\ell(\la_{\downarrow(i,j)})}.
\end{align}
As a result 
\begin{align*}
\nonumber \al^{\ell(\la)}z_\la[p_\la]\left([D(\al),E_2]J_\gamma^\al\right)&= \\
&  \sum_{i}im_i(\la)\left[(\al-1)(i-1)\al^{\ell(\la_{\downarrow(i)})}z_{\la_{\downarrow(i)}}\theta^{\gamma}_{\la_{\downarrow(i)}}\right.\\
&+\sum_{d=1}^{i-2} z_{\la^{\uparrow(i-1-d,d)}}\al^{\ell(\la^{\uparrow(i-1-d,d)})}\theta^{\gamma}_{\la^{\uparrow(i-1-d,d)}}\\ &+\al\sum_{j}\left.(m_j(\la)-\delta_{ij})z_{\la_{\downarrow(i,j)}}\al^{\ell(\la_{\downarrow(i,j)})}\theta^{\gamma}_{\la_{\downarrow(i,j)}}\right].
\end{align*}
The substitution of the above value for $\al^{\ell(\la)}z_\la[p_\la]\left([D(\al),E_2]J_\gamma^\al\right)$ in Equation (\ref{eq : lala}) gives the desired recurrence relation for Jack connection coefficients.

%%%%%%%%%%%%%%%%%%%%%%%%%%%%%%%%%%%%%%%%%%%%%%%%%%%%%%%%%%%%%%%%%%%%%%

\section{Proof of the Matchings-Jack conjecture}
\label{sec : add}
\subsection{Proof of Theorem \ref{thm : reca}}

In order to prove Theorem \ref{thm : reca} we need to show that
\begin{equation}
 (\al-1)(\la_i-1)a_{nn}^{\la_{\downarrow(\la_i)}}(\al)+\sum_{d=1}^{\la_i-2}a_{nn}^{\la^{\uparrow(\la_i-1-d,d)}}(\al)+\al\sum_{j\neq i}\la_ja_{nn}^{\la_{\downarrow(\la_i,\la_j)}}(\al)
\end{equation}
is independent of the choice of integer $i\in\{1,\ldots, \ell(\la)\}$.\\ %and as an immediate consequence

To this extent we introduce a few more notations. Consider the formal vector space over the polynomials in $\al$ spanned by the set $\{a_{nn}^\la(\al)\}_{n \geq 1,\la \vdash n}$. We assume that Equation (\ref{eq : recnu}) holds for $\nu =(n)$. For integers $i$, $d$ ($i \geq d+2$) and $j$ define the linear operators $\Xi_i$, $\Omega_{i,d}$ and $\Phi_{i,j}$ on the $a_{nn}^\la(\al)$ for all partitions $\la$ such that $i,j \in \la$ by:
\begin{align*}
&\Xi_ia_{nn}^\la(\al) = a_{n-1,n-1}^{\la_{\downarrow(i)}}(\al),\\
&\Omega_{i,d}a_{nn}^\la(\al) = a_{n-1,n-1}^{\la^{\uparrow(i-1-d,d)}}(\al),\\
&\Phi_{i,j}a_{nn}^\la(\al) =a_{n-1,n-1}^{\la_{\downarrow(i,j)}}(\al).
\end{align*}
%(m_j(\la)-\delta_{ij})
Let also for $i\in \la$
\begin{align*}
&\Theta_{i} a_{nn}^\la(\al)=\left((\al-1)(i-1)\Xi_i+\sum_{d=1}^{i-2}\Omega_{i,d}+\al\sum_{j}j(m_j(\la)-\delta_{ij})\Phi_{i,j}\right)a_{nn}^\la(\al).
\end{align*}
We have the following proposition.
\begin{prop}
\label{lem : TiTr} Consider the operators defined above. For integers $i$ and $r$, the following identity is true on the subspace spanned by the $\{a_{nn}^\la(\al)\}_{n \geq 1,\la \vdash n, i,r\in \la}$
\begin{equation}
\label{eq : TiTr}\Theta_{r}(\Theta_{i}-\al r\Phi_{i,r}) = \Theta_{i}(\Theta_{r}-\al i\Phi_{i,r}) + \al(i-r)\Theta_{i+r-1}\Phi_{i,r}.  
\end{equation}
\end{prop}
\begin{proof}
On can write $\Theta_i$ as $\Theta_i=C_i+\Sigma_i$ where $$C_i = (\al-1)(i-1)\Xi_i+\sum_{d=1}^{i-2}\Omega_{i,d}$$ and $$\Sigma_i =\al\sum_{j}j(m_j(\la)-\delta_{ij})\Phi_{i,j}$$ when applied to $a_{nn}^\la(\al)$. The $(C_i)_{i\geq1}$ commute between each other while the $(\Sigma_i)_{i\geq 1}$ do not. The LHS of Equation (\ref{eq : TiTr}) reads:
\begin{equation*}
\Theta_{r}(\Theta_{i}-\al r\Phi_{i,r}) = (C_r + \Sigma_r)(C_i + \Sigma_i-\al r\Phi_{i,r}).
\end{equation*}
One can check the following identities directly:
\begin{align}
&\label{eq : SiSi} \Sigma_r(\Sigma_i-\al r\Phi_{i,r}) = \Sigma_i(\Sigma_r-\al i\Phi_{i,r})+\al(i-r)\Sigma_{i+r-1}\Phi_{i,r},\\
&\nonumber \Sigma_rC_i = C_i(\Sigma_r-\al i \Phi_{i,r})+\al\left((\al-1)(i-1)^2\Xi_{i+r-1}+\sum_{d=1}^{i-2}(i-1-d)\Omega_{i+r-1,d}\right.\\
&\label{eq : SiC}\left.\phantom{\Sigma_rC_i = C_i(\Sigma_r-\al i \Phi_{i,r})+\al(\al-1)}-\sum_{d=r-1}^{i+r-3}(r-1-d)\Omega_{i+r-1,d}\right)\Phi_{i,r}.
\end{align}
%\end{lem}
\noindent Applying Equation (\ref{eq : SiC}) to $\Sigma_rC_i$ and to $ C_r(\Sigma_i-\al r \Phi_{i,r})$, yields
\begin{align}
\nonumber\Sigma_rC_i+&C_r(\Sigma_i-\al r \Phi_{i,r}) =C_i(\Sigma_r-\al i \Phi_{i,r})+\Sigma_iC_r \\
\nonumber&+ \al\left((\al-1)[(i-1)^2-(r-1)^2]\Xi_{i+r-1}+\sum_{d=1}^{i+r-3}(i-r)\Omega_{i+r-1,d}\right)\Phi_{i,r}\\
\label{eq : SiC2} \phantom{\Sigma_rC_i+}&\phantom{C_r(\Sigma_i-\al r \Phi_{i,r})} = C_i(\Sigma_r-\al i \Phi_{i,r})+\Sigma_iC_r +\al(i-r)C_{i+r-1}\Phi_{i,r}.
\end{align}
As a result,
\begin{align*}
(C_r + \Sigma_r&)(C_i + \Sigma_i-\al r\Phi_{i,r}) \\
&=C_rC_i + C_r(\Sigma_i-\al r\Phi_{i,r})+ \Sigma_rC_i+\Sigma_r( \Sigma_i-\al r\Phi_{i,r})\\
&\stackrel{(\ref{eq : SiSi}), (\ref{eq : SiC2})}=C_iC_r+C_i(\Sigma_r-\al i \Phi_{i,r})+\Sigma_iC_r +\al(i-r)C_{i+r-1}\Phi_{i,r}\\
&\phantom{C_iC_r+C_i(\Sigma}+\Sigma_i(\Sigma_r-\al i\Phi_{i,r})+\al(i-r)\Sigma_{i+r-1}\Phi_{i,r}\\
&=(C_i + \Sigma_i)(C_r + \Sigma_r-\al i\Phi_{i,r}) + \al(i-r)(C_{i+r-1}+\Sigma_{i+r-1})\Phi_{i,r}.
\end{align*}
This is the desired formula.
\end{proof}
%Theorem \ref{thm : albe} is equivalent to the following statement.
\noindent To prove Theorem \ref{thm : reca}, we need to show the following equivalent statement.
\begin{thm}
For all integers $i,r$ and partition $\la \vdash n$ containing $i$ and $r$
\begin{equation}
\Theta_i a_{nn}^\la(\al)= \Theta_r a_{nn}^\la(\al) = a_{nn}^\la(\al).
\end{equation}
\end{thm}
\begin{proof} 
By recurrence on $n$. The property is trivial for $n=2$.
Assume $\Theta_ia_{k,k}^\mu(\al) = a_{k,k}^\mu(\al)$ for all $2\leq k \leq n$ whenever $i\in \mu\vdash k$. Then for partition $\la \vdash n+1$ containing at least one part $r$ and one part $i$, we have
\begin{equation*}
\Theta_{i}a_{n+1,n+1}^\la(\al) = (\Theta_{i}-\al r\Phi_{i,r})a_{n+1,n+1}^\la(\al)+\al r\Phi_{i,r}a_{n+1,n+1}^\la(\al).
\end{equation*}
According to the recurrence hypothesis, as $(\Theta_{i}-\al r\Phi_{i,r})a_{n+1,n+1}^\la(\al)$ depends on the $a_{nn}^\mu(\al)$ with $r\in\mu$, we have
\begin{equation*}
(\Theta_{i}-\al r\Phi_{i,r})a_{n+1,n+1}^\la(\al) = \Theta_{r}(\Theta_{i}-\al r\Phi_{i,r})a_{n+1,n+1}^\la(\al).
\end{equation*}
Then in view of Proposition \ref{lem : TiTr}
\begin{align*}
\Theta_{i}a_{n+1,n+1}^\la(\al) = (\Theta_{i}(\Theta_{r}-\al i\Phi_{i,r}) &+ \al(i-r)\Theta_{i+r-1}\Phi_{i,r})a_{n+1,n+1}^\la(\al)\\
&+\al r\Phi_{i,r}a_{n+1,n+1}^\la(\al).
\end{align*}
Notice that $\Phi_{i,r}a_{n+1,n+1}^\la(\al) = a_{nn}^{\la_{\downarrow(i,r)}}(\al)$ and $\la_{\downarrow(i,r)}$ contains $i+r-1$. Applying the recurrence hypothesis to $\Theta_{i}$ and $\Theta_{i+r-1}$ gives
\begin{align*}
\Theta_{i}a_{n+1,n+1}^\la(\al) &= ((\Theta_{r}-\al i\Phi_{i,r})+ \al(i-r)\Phi_{i,r}+\al r\Phi_{i,r})a_{n+1,n+1}^\la(\al)\\
&=\Theta_{r}a_{n+1,n+1}^\la(\al).
\end{align*}
Finally, using Equation (\ref{eq : recnu}) in the case $\nu = (n)$
\begin{align*}
(n+1)a_{n+1,n+1}^\la(\al) &=\sum_{i}im_i(\la)\Theta_ia_{n+1,n+1}^\la(\al)\\
&=\left(\sum_{i}im_i(\la)\right)\Theta_{r}a_{n+1,n+1}^\la(\al)\\
&=(n+1)\Theta_{r}a_{n+1,n+1}^\la(\al).
\end{align*}
\end{proof}

\subsection{Matching Jack Conjecture in the case $\mu=\nu=(n)$}

In this section we prove Theorem \ref{thm : MJC}. 

\begin{proof}
We define the functions $(\wt_\la)_{\la\vdash n}$ and prove the following statements 
\begin{align} &\label{MJC1} a^{\la}_{nn}(\be+1)=\sum_{\de\in\mathcal{G}(\la)} \be^{\wt_{\la}(\de)}, \\ 
&\label{MJC2}\wt_\la(\de)\in \{0,1,\ldots,n-1\}, \\
&\label{MJC3}\wt_{\la}(\de)=0 \iff \de \text{ is bipartite,}\\
&\label{MJC4}[\be^{n-1}]a^\la_{nn}(\be+1) = (n-1)! \end{align}
for all  $\la\vdash n$ and $\de\in \mathcal{G}(\la)$ by induction on $n$. 
\vskip 3 pt
{\bf Base.} Define $\wt_{(1)}(\{1,\widehat{1}\})=0$. As $a^{(1)}_{11}(\be+1)=1=\be^0$ then statements (\ref{MJC1})--(\ref{MJC4})
obviously hold for $n=1$.
\vskip 3 pt
{\bf Step.} Fix some $N\geq 1$ and suppose that $\wt_\la(\de)$ is defined for all $\la$ with $|\la|\leq N$ and $\de\in \mathcal{G}(\la)$ in such way that assertions (\ref{MJC1})--(\ref{MJC4}) are true for $n\in\{1,\ldots,N\}$. Now fix some partition $\la\vdash N+1$, a canonically labelled $\la$-graph $G$ and a matching $\de\in \mathcal{G}(\la)$. In~order to define the value $\wt_{\la}(\de)$ write $\de=\{\{1,v\}\}\cup \de'$ and replace the edges of $G$ associated with vertices $1$ and $v$ as in Section \ref{sec : red}. We get a $\la'$-graph $G'_v$ where 
\begin{equation} \label{cases_la} \la'=\left\{ \begin{aligned} 
&\la_{\downarrow(\la_1)} & \text{ if } & v \in \{2,\ldots,\la_1\},\\ 
&\la^{\uparrow(d,\la_1-1-d)} & \text{ if } & v=\widehat{d} \text{ where } d \in \{1,\ldots,\la_1-2\},\\ 
&\la_{\downarrow(\la_1, \la_j)}& \text{ if } & v\in C_{2\la_j} \text{ ($2\la_j$-cycle of $G$) where $j>1$}.\end{aligned}\right.
\end{equation}
 
We renumber vertices of $G'_v$ by labels $[N]\cup[\widehat{N}]$ canonically in any way, and besides, if $v\in [\widehat{N+1}]$ then we preserve the types of vertices (with/without hat). Denote by $\mathcal{G}_v(\la')$
 the set of good matchings of the $\la'$-graph $G'_v$. It is clear that
 $$\begin{aligned} \de\in \mathcal{G}(\la)\; &\iff& &\de'\in \mathcal{G}_v(\la'),\\ \de \text{ is bipartite }\; &\iff& &\de' \text{ is bipartite.}\end{aligned}$$ 

The value $\wt_{\la'}(\de')$ is well defined by the inductive assumption. Define  
\begin{equation} \label{defwt}
\wt_{\la}(\de):=\left\{ \begin{aligned} &\wt_{\la'}(\de')+1 & \text { if } & v \in [N+1],\\
&\wt_{\la'}(\de') & \text { if } & v \in [\widehat{N+1}].\end{aligned}\right.\end{equation}

It is obvious that statements (\ref{MJC2}) and (\ref{MJC3}) are true for $n=N+1$.

In order to prove the main equality (\ref{MJC1}) for $n=N+1$ we split the sum 
$$\sum_{\de\in\mathcal{G}(\la)} \be^{\wt_{\la}(\de)}=\sum_{v\in [N+1]}\sum_{\de'\in \mathcal{G}_v(\la')} \be^{\wt_{\la'}(\de')+1}+\sum_{v\in [\widehat{N+1}]}\sum_{\de'\in \mathcal{G}_v(\la')} \be^{\wt_{\la'}(\de')}$$ into four parts according to (\ref{cases_la}):
\begin{multline} \label{MJCstep} \sum_{\de\in\mathcal{G}(\la)}\be^{\wt_{\la}(\de)} =
\left[ \sum_{v=2}^{\la_1}\sum_{\de'\in \mathcal{G}_v(\la_{\downarrow(\la_1)})}+\sum_{d=1}^{\la_1-2}\sum_{\de'\in \mathcal{G}_{\widehat{d}}(\la^{\uparrow(d,\la_1-1-d)})}+\right. \\ \left.+\sum_{j>1}\sum_{v\in C_{2\la_j}\cap[N+1]}\left(\sum_{\de'\in \mathcal{G}_v(\la_{\downarrow(\la_1, \la_j)})}+\sum_{\de'\in \mathcal{G}_{\widehat{v}}(\la_{\downarrow(\la_1, \la_j)})}\right)\right]\be^{\wt_{\la}(\de)}=\\ \stackrel{(\ref{defwt})}=(\la_1-1)\be\!\!\sum_{\de'\in \mathcal{G}(\la_{\downarrow(\la_1)})}\!\!\be^{\wt_{\la_{\downarrow(\la_1)}}(\de')}+\sum_{d=1}^{\la_1-2}\sum_{\de'\in \mathcal{G}(\la^{\uparrow(d,\la_1-1-d)})}\!\!\!\be^{\wt_{\la^{\uparrow(d,\la_1-1-d)}}(\de')}+\\+\sum_{j>1}\la_j(\be+1)\sum_{\de'\in \mathcal{G}(\la_{\downarrow(\la_1, \la_j)})} \be^{\wt_{\la_{\downarrow(\la_1, \la_j)}}(\de')}. \end{multline}

We see that the RHS of (\ref{MJC1}) satisfies the recurrence relation of Formula (\ref{eq : rec}). So equality (\ref{MJC1}) is proved.

Finally we prove equality (\ref{MJC4}) for $n=N+1$. By inductive assumption we have 
$$[\be^{N-1}]a^{\la'}_{NN}(\be+1) = (N-1)!$$ for all $\la'\vdash N$. So equlities (\ref{MJCstep}) show that for $\la\vdash N+1$
$$[\be^{N}]a^\la_{N+1,N+1}(\be+1) = \left((\la_1-1)+\sum_{j>1} \la_j\right)(N-1)!=N!.$$
Theorem \ref{thm : MJC} is proved.
\end{proof}

\begin{rem} Our definition of the function $\wt_{\la}$ is based on the found recurrence formula. Note that the value $\wt_{\la}(\de)$ depends on the choice of the starting vertex $1$ in general if $\de$ is not bipartite. For instance, the three non bipartite matchings of a $(3)$-graph (see Fig.~\ref{good3}) are equivalent as graphs, but one of them has weight $1$ and the two other ones have weight $2$ (depending on the choice of the starting vertex). Recompute $a^{3}_{33}(\beta+1)=2\be^2+\be+1$ (see examples in Section \ref{ccc}).    
\end{rem}

\subsection{Proof of Theorem \ref{thm : GenCoeff}}
Denote by $\Gamma^{l}_n$ the generating series defined as
\begin{equation*}
\Gamma^{l}_n= \sum_{\gamma \vdash n}\frac{(\theta_n^\gamma(\al))^l}{j_\gamma(\al)}J_\gamma^\al.
\end{equation*}
One can easily generalize the equation of Theorem \ref{thm : JJJ} as follows:
\begin{equation*}
\Gamma^{l}_n = \frac{1}{n}\Delta_l(\al)\left (\Gamma^{l}_{n-1}\right).
\end{equation*}
\begin{proof}
The proof is similar to the one of Section \ref{sec : rtsfjsf}. In Equation (\ref{eq : 40}) replace the multiplication by $\theta_{n+1}^\rho(\al)$ by a multiplication by $(\theta_{n+1}^\rho(\al))^{l-1}$. Then Lemma \ref{lem : theta} is applied $l-1$ times instead of once so that Equation (\ref{eq : Ti-T}) becomes:
\begin{align*}
\nonumber &\sum_{\rho \vdash n+1}\frac{(\theta^\rho_{n+1}(\al))^{l-1}J_\rho^\al(x) E_2^\perp J_\rho^\al(y)}{j_\rho(\al)} =\\
\nonumber&=\frac{1}{\al}\sum_{\gamma \vdash n}\frac{\theta_{n}^\gamma(\al)J_{\gamma}^\al(y)}{j_{\gamma}(\al)} \sum_{i=1}^{\ell(\gamma)+1}c_i(\gamma)\left(\theta^{\gamma^{(i)}}_{[1^{n-1}2]}(\al)-\theta^{\gamma}_{[1^{n-2}2]}(\al)\right)^lJ_{\gamma^{(i)}}^\al(x).
\end{align*}
The proof is achieved by noticing that 
\begin{align}
\nonumber \Delta_{l}(\al) = [D(\al),[\ldots,[D(\al),p_1/\al]\ldots]] \\ \label{eq : DeltaL} =  \frac{1}{\al}\sum_{k}\binom{l}{k}(-1)^{l-k}D(\al)^kp_1D(\al)^{l-k}, 
\end{align}
using Equation (\ref{eq : DapDb}) and extracting the coefficients in $p_n(y)$. 
\end{proof}
As a consequence, one has
\begin{equation}
\Gamma^{l}_n(x) = \frac{1}{n!}\Delta_l(\al)^{n-1}\left (\Gamma^{l}_{1}(x)\right).
\end{equation}
But clearly, $\Gamma^{l}_{1}(x) = p_1/\al$. 
The next step is to notice that
\begin{align*}
\sum_\la \al^{-\ell(\la)}z_\la^{-1}a_\la^{l,r}(\al)p_\la &= \sum_\gamma\frac{(\theta_{[1^{n-2}2]}^\gamma(\al))^r(\theta_n^\gamma(\al))^l}{j_\gamma(\al)}J_\gamma^\al\\
 &= D(\al)^r\left(\sum_\gamma\frac{(\theta_n^\gamma(\al))^l}{j_\gamma(\al)}J_\gamma^\al\right)\\
 &=\frac{1}{n!}D(\al)^r\Delta_l(\al)^{n-1}(p_1/\al).
\end{align*}
To prove the polynomial properties of Theorem \ref{thm : GenCoeff}, first notice that Equation (\ref{eq : DeltaL}) can be rewritten as
\begin{align}
\nonumber \Delta_{l}(\al) &= [D(\al),[\ldots,[D(\al),E_2]\ldots]]\\ \label{eq : DeltaLE} & = \sum_{k}\binom{l-1}{k}(-1)^{l-1-k}D(\al)^kE_2D(\al)^{l-1-k}, 
\end{align}
The combination of Equation (\ref{eq : DeltaLE}) and Lemma \ref{lem : dp} shows that the coefficients in the power sum expansion of $D(\al)^r\Delta_{l}^{n-1}(\al)(p_1)$ are polynomials in $\al$ with integer coefficients of degree at most $(n-1)(l-1)+r$ for $n\geq2$. Denote
\begin{equation*}
\al^{-\ell(\la)}z_\la^{-1}a_\la^{l,r} (\al)=\frac{1}{\al n!} \sum_{i=0}^{(n-1)(l-1)+r}g_i\al^i.
\end{equation*}
The coefficients $g_i$ are integers. Using \cite[Thm. 5]{V2014}, we have
\begin{equation*}
\al(-\al)^{(l+r-1)n+r(1-n)-l-\ell(\la)}\sum_{i=0}^{(n-1)(l-1)+r}g_i\al^{-i} =\frac{1}{\al} \sum_{j=0}^{(n-1)(l-1)+r}g_j\al^{j}.
\end{equation*} 
Equating the coefficients in $\al^j$ yields:
\begin{equation*}
(-1)^{(l-1)(n-1)+r-1-\ell(\la)}g_{(l-1)(n-1)+r+1-\ell(\la)-j} = g_j.
\end{equation*}
But the $g_i$ are non zero only for non negative $i$. As a result $g_i = 0$ for $i\geq (l-1)(n-1)+r+2-\ell(\la)$. We write
\begin{equation*}
\al^{-\ell(\la)}z_\la^{-1}a_\la^{l,r} (\al) = \frac{1}{\al n!} \sum_{i=0}^{(n-1)(l-1)+r+1-\ell(\la)}g_i\al^i.
\end{equation*}
Finally
\begin{equation*}
|C_\la|a_\la^{l,r} (\al) = \sum_{i=0}^{(n-1)(l-1)+r+1-\ell(\la)}g_i\al^{i+\ell(\la)-1} =\sum_{i=\ell(\la)-1}^{(n-1)(l-1)+r}g_{i-\ell(\la)+1}\al^i
\end{equation*}
is a polynomial in $\al$ of degree at most $(n-1)(l-1)+r$. Furthermore:
\begin{equation*}
[\al^i]|C_\la|a_\la^{l,r} (\al) = (-1)^{(l-1)(n-1)+r+\ell(\la)-1}[\al^{(l-1)(n-1)+r+\ell(\la)-1-i}]|C_\la|a_\la^{l,r} (\al).
\end{equation*}
%%%%%%%%%%%%%%%%%%%%%%%%%%%%%%%%%%%%%%%%%%%%%%%%%%%%%%%%%%%%%%%%%%%%%%%%%%%%%%%%%%%%%%%%%%%%%%%%%%%%%%%%%%%%%%%%%%%

\end{document}